\documentclass[12pt,reqno]{amsart}
\usepackage{amsfonts}
\usepackage{amssymb,color}
\usepackage{amssymb}

\usepackage[margin=3cm, a4paper]{geometry}

\def\normo#1{\left\|#1\right\|}

\def\aabs#1{\left|#1\right|}
\def\brk#1{\left(#1\right)}

\def\half#1{\frac{#1}{2}}
\def\norm#1{\|#1\|}

\def\wt#1{\widetilde{#1}}

\newcommand{\T}{{\mathbb T}}

\newcommand{\R}{{\mathbb R}}
\newcommand{\C}{{\mathbb C}}
\newcommand{\Z}{{\mathbb Z}}
\newcommand{\ft}{{\mathcal{F}}}

\newcommand{\les}{{\lesssim}}
\newcommand{\ges}{{\gtrsim}}

\newcommand{\sgn}{{\mbox{sgn}}}

\def\half#1{\frac{#1}{2}}

\def\norm#1{\|#1\|}
\def\normo#1{\left\|#1\right\|}

\def\wt#1{\widetilde{#1}}

\def\aabs#1{\left|#1\right|}
\def\ve#1{\mathbf{#1}}

\newcommand{\F}{\mathcal{F}}

\newcommand{\cir}{\mathbb{S}}

\newcommand{\e}{\varepsilon}

\newcommand{\Des}{\Delta_\theta}

\newcommand{\p}{\partial}

\newcommand{\I}{\infty}

\newcommand{\Del}[1]{}

\numberwithin{equation}{section}

\newtheorem{thm}{Theorem}[section]

\newtheorem{lem}[thm]{Lemma}

\theoremstyle{remark}
\newtheorem{rem}{Remark}

\theoremstyle{remark}

\theoremstyle{definition}

\begin{document}

\title[2D cubic derivative Schr\"odinger equation]{Spherically averaged maximal function and scattering for the 2D cubic derivative Schr\"odinger equation
}
\author[Z. Guo]{Zihua Guo}
\address{School of Mathematical Sciences, Monash University, VIC 3800, Australia \& LMAM, School of Mathematical Sciences, Peking
University, Beijing 100871, China}
\email{zihua.guo@monash.edu}

\begin{abstract}
We prove scattering for the 2D cubic derivative Schr\"odinger
equation with small data in the critical Besov space with one degree angular regularity.
The main new ingredient is that we
prove a spherically averaged maximal function estimate for the 2D
Schr\"odinger equation. We also prove a global well-posedness result
for the 2D Schr\"odinger map in the critical Besov space with one degree angular regularity.
The key ingredients for the latter results are the spherically averaged maximal function estimate,
null form structure observed in \cite{Bej}, as well as the
generalised spherically averaged Strichartz estimates obtained in \cite{Guo2} in order to exploit the null form structure.
\end{abstract}

\maketitle

\tableofcontents

\section{Introduction}

In this paper, we study the Cauchy problem for the cubic derivative
Schr\"odinger equation
\begin{align}\label{eq:cDNLS}
\begin{cases}
i\partial_tu+\Delta u=|u|^2\ve{a}\cdot \nabla u+u^2\ve{b}\cdot \nabla \bar u,\\
u(x,0)=u_0(x)
\end{cases}
\end{align}
where $u(x,t):\R^n\times\R \to \C$, $\ve{a}, \ve{b}\in \C^2$. The
equation \eqref{eq:cDNLS} arises from the strongly interacting
many-body systems near criticality as recently described in terms of
nonlinear dynamics \cite{CT}. The Schr\"odinger equation with
derivative in the nonlinearity of the form
\begin{align}\label{eq:DNLS}
i\partial_tu+\Delta u=F(u,\bar u, \nabla u,\nabla \bar u),
\end{align}
has been studied extensively, e.g. see the introduction of
\cite{KPV,Wang1} for the history of the study. Besides equation
\eqref{eq:cDNLS} and the well-known one-dimension derivative
Schr\"odinger equation, \eqref{eq:DNLS} contains another important
model known as the Schr\"odinger maps
\begin{align}\label{eq:Schmap}
\p_t s=s\times \Delta_x s,\quad s(0)=s_0,
\end{align}
where $s:\R^n\times \R\to \cir^2\hookrightarrow \R^3$. It was known
that under stereographic projection (see Section 5 below)
\eqref{eq:Schmap} is equivalent to
\begin{align}\label{eq:Schmap2}
i\partial_tu+\Delta u=\frac{2\bar
u}{1+|u|^2}\sum_{i=1}^n(\p_{x_i}u)^2,
\end{align}
and we see that the cubic term $\bar u (\p_{x_i}u)^2$ serves as the
first term in the Taylor expansion of the above nonlinear term.

Note that the equation \eqref{eq:cDNLS} is invariant under the
following scaling transform: for $\lambda>0$
\[u(x,t)\to \lambda^{1/2} u(\lambda x,\lambda^2t),\quad u_0(x)\to \lambda^{1/2} u_0(\lambda x),\]
then we see that the critical space for \eqref{eq:cDNLS} is $\dot
H^{\frac{n-1}{2}}$ in the sense of scaling. Because of the loss of
the derivative, the usual Strichartz analysis as for the power type
nonlinearity doesn't work for \eqref{eq:cDNLS}. One needs some
estimates with stronger smoothing effect. Kenig, Ponce, and Vega
\cite{KPV} introduced for the first time a method to obtain local
well-posedness for general derivative Schr\"odinger equations. This
method combines ``local smoothing estimates", ``inhomogeneous local
smoothing estimates", which give the crucial gain of one derivative,
and ``maximal function estimates". In the study of Schr\"odinger
map, Ionescu and Kenig \cite{IK,IK2} introduced the anisotropic
local smoothing and maximal function estimates for Schr\"odinger
equation. It was proved in \cite{IK2} that the following local
smoothing estimates hold
\begin{align*}
\norm{e^{it\Delta}P_{k,\ve e_1}f}_{L_{x_1}^\infty L_{\bar x,
t}^2}\les& 2^{-k/2}\norm{f}_2,\\
\normo{\int_0^te^{i(t-s)\Delta}P_{k,\ve e_1}g}_{L_{x_1}^\infty
L_{\bar x, t}^2}\les& 2^{-k}\norm{g}_{L_{x_1}^1 L_{\bar x, t}^2},
\end{align*}
where $P_{k,\ve e_1}f=\ft^{-1}1_{|\xi_1|\sim 2^k}1_{|\xi|\sim
2^k}\hat f$ (roughly, see Section 2 for the definition). In order to
apply these estimates to deal with the cubic nonlinear terms, the
following maximal function appears naturally
\begin{align}\label{eq:max}
\norm{e^{it\Delta}P_{k}f}_{L_{x_1}^2 L_{\bar x, t}^\infty}\les&
2^{(n-1)k/2}\norm{f}_2.
\end{align}
It was proved in \cite{IK2} that \eqref{eq:max} holds if $n\geq 3$.
These estimates played key roles in the consequent study of
Schr\"odinger map, e.g. in \cite{BIKT}. For \eqref{eq:cDNLS}, in
three dimensions and higher, one could gather these estimates to
obtain global well-posedness and scattering in the critical Besov
space, see \cite{ywang} which also generalized the estimates and
results to the non-elliptic case. In \cite{Wang1,Wang2} these
estimates were generalized to the modulation space and sharp global
well-posedness in modulation spaces for \eqref{eq:cDNLS} (also in
the non-elliptic case) with $n\geq 3$ were obtained.

However, if $n=2$, \eqref{eq:max} fails. Thus the cubic nonlinear
term in two dimensions is more difficult. To the author's knowledge,
there are two approaches to deal with this difficulty. The first one
was developed in \cite{BIKT} which uses the Galilean invariance of
the Schr\"odinger propagator. The space $L_{x_1}^2L_{x_2,t}^\infty$
is replaced with a sum of Galilean transforms of it. The idea of
using such sums of spaces as substitutes is due to Tataru
\cite{Tataru}. The space is defined for any finite time interval
$[-T,T]$, but the estimates in \cite{BIKT} are independent of $T$.
The second one was developed in \cite{Wang1} which proves the
following estimate via the Gabor frame representation of linear
Schr\"odinger solution: for $1\leq r<2$
\begin{align*}
\norm{e^{it\Delta}P_{0}f}_{L_{x_1}^2 L_{\bar x, t}^\infty}\les&
\norm{f}_r.
\end{align*}
Then with this well-posedness and scattering for \eqref{eq:cDNLS}
with $n=2$ were proved for suitable data in some modulation space.

In this paper, we take an another approach. Our ideas are inspired
by \cite{Tao} and the recent work \cite{Guo}. First, since
\eqref{eq:max} only fails ``logarithmically" for $n=2$, we find that
the spherically averaged maximal function estimate holds. This is
like the spherically averaged endpoint Strichartz estimate for the
2D Schr\"odinger equation that was studied in \cite{Tao}. Note that,
the Strichartz space $L_t^2L_x^\infty$ is rotational invariant,
however the anisotropic space $L_{x_1}^2 L_{\bar x, t}^\infty$ is
not. It is a bit surprising that we have
\begin{thm}\label{thm:max}
There exists $C>0$ such that for $k\in \Z$, $u_0\in L^2(\R^2)$, one
has
\begin{align}
\norm{e^{it\Delta}P_k u_0}_{L^2_{x_1}L_{x_2,t}^\infty L_\theta^2}
\leq C 2^{k/2}\norm{u_0}_2.
\end{align}
\end{thm}

See Section 2 for the definition of the space
$L^2_{x_1}L_{x_2,t}^\infty L_\theta^2$ and $P_k$. Then we use an
argument of \cite{ywang} (which is in the spirit of \cite{BIKT}) to
derive the corresponding inhomogeneous estimate. To use this norm to
the equation \eqref{eq:cDNLS}, as in \cite{Guo} we assume sufficient
regularity on the sphere variable such that the space on the sphere
is an algebra. Not like the Strichartz space, the local
smoothing/maximal function space is anisotropic in $x$ which makes
it not very compatible with the spherical average. For example, we
do not have compare between $L^2_{x_1}L_{x_2,t}^\infty L_\theta^2$
and $L^2_{x_1}L_{x_2,t}^\infty$. Fortunately, we can still close the
iteration arguments in these spaces. We show

\begin{thm}\label{thm:2cDNLS}
Assume $n=2$, $u_0\in \dot B_{2,1,\theta}^{1/2,1}$ with
$\norm{u_0}_{\dot B_{2,1,\theta}^{1/2,1}}=\e_0\ll 1$. Then there
exists a unique global solution $u$ to \eqref{eq:cDNLS} such that
$\norm{u}_{F^{1/2}}\les \e_0$. Moreover, the map $u_0\to u$ is
Lipshitz from $\dot B_{2,1,\theta}^{1/2,1}$ to $C(\R;\dot
B_{2,1,\theta}^{1/2,1})$, and scattering holds in this space.
\end{thm}
\begin{rem}
In Theorem \ref{thm:2cDNLS}, $u_0\in \dot B_{2,1,\theta}^{1/2,1}$
means that $u_0\in \dot B_{2,1}^{1/2}$ and its spherical derivative
$\partial_\theta u_0\in \dot B_{2,1}^{1/2}$, and $\norm{u_0}_{\dot
B_{2,1,\theta}^{1/2,1}}=\norm{u_0}_{\dot
B_{2,1}^{1/2}}+\norm{\p_\theta u_0}_{\dot B_{2,1}^{1/2}}$. We do not
need $X^{s,b}$-structure for the proof of Theorem \ref{thm:2cDNLS},
and see Section 4 for the definition of
$F^s$. Note that if $u_0\in \dot B_{2,1}^{1/2}$ is radial, then $u_0
\in \dot B_{2,1,\theta}^{1/2,1}$. This is a bit surprising that for
radial data the problem is relatively simpler even though the radial
symmetry is not preserved under the flow of \eqref{eq:cDNLS}.
\end{rem}

Now we turn to the study of the Schr\"odinger map. It has also been
studied extensively (also in the case in which the sphere $\cir^2$
is replaced by more general targets). Based on variants of the
energy method, the local existence of the sufficiently smooth
solutions were obtained, even for large data, see, for example,
\cite{SSB,CS,Ding,KPST} and the references therein. Similarly as
\eqref{eq:cDNLS}, by the scaling we see the critical space for
\eqref{eq:Schmap} is $\dot H^{d/2}$. Local well-posedness were
obtained \cite{IK} for small data in $H^s_Q(\R^n:\cir^2)$,
$s>(n+1)/2$. This was improved to $s>n/2$ by Bejenaru \cite{Bej}.
Bejenaru observed for the first time in the setting of Schr\"odinger
maps, that the gradient part of the nonlinearity in
\eqref{eq:Schmap2} has a certain null structure. Global
well-posedness for small data in the critical Besov space in
dimensions $n\geq 3$ were obtained in \cite{IK2}, and independently
in \cite{Bej2}. Recently, global well-posedness for small data in
the critical Sobolev space were proved in \cite{BIK} first for
$n\geq 4$, and in \cite{BIKT} for $n\geq 2$ where some state of art
techniques were built. We revisit the case $n=2$ using the new
maximal function estimate. We prove

\begin{thm}\label{thm:Schmap}
Assume $n=2$, the Schr\"odinger map initial value problem
\eqref{eq:Schmap} is globally well-posed for small data $s_0\in
\dot{B}^{1,1}_Q(\R^2;\cir^2)$, $Q\in \cir^2$.
\end{thm}

\begin{rem}
The space $\dot{B}^{s,1}_Q$ is defined by
\[\dot{B}^{s,1}_Q=\{f: \R^2\to \R^3; f-Q\in \dot{B}^{s,1}_{2,1,\theta},\, |f(x)|\equiv 1 \mbox{ a.e. in } \R^2\}.\]
In the proof of Theorem \ref{thm:Schmap}, we will use $X^{s,b}$-type
space in order to exploit the null structure as in \cite{Bej,IK}. In two dimension, there is a logarithmic
problem to exploit the null structure which does not appear in 3D and higher. Fortunately, we can use the generalised
spherically averaged Strichartz estimates obtained in \cite{Guo2} to overcome it.  So the additional angular regularity is not only needed for the new maximal function, but also for using null structure.
\end{rem}

\section{Definitions and Notations}

For $x, y\in \R$, $x\les y$ means that there exists a constant $C$
such that $x\leq Cy$, and $x\sim y$ means that $x\les y$ and $y\les
x$. We use $\ft(f)$, $\hat{f}$ to denote the space-time Fourier
transform of $f$, and $\ft_{x_i,t}f$ to denote the Fourier transform
with respect to $x_i,t$.

Let $\eta: \R\to [0, 1]$ be an even, non-negative, radially
decreasing smooth function such that: a) $\eta$ is compactly
supported in $\{\xi:|\xi|\leq 8/5\}$; b) $\eta\equiv 1$ for
$|\xi|\leq 5/4$. For $k\in \Z$ let
$\chi_k(\xi)=\eta(\xi/2^k)-\eta(\xi/2^{k-1})$ and $\chi_{\leq
k}(\xi)=\eta(\xi/2^k)$, $\widetilde
\chi_k(\xi)=\sum_{l=-9}^9\chi_{k+l}(\xi)$ and then define the
Littlewood-Paley projectors $P_k, P_{\leq k}, P_{\geq k}$ on $L^2(\R^2)$ by
\[\widehat{P_ku}(\xi)=\chi_k(|\xi|)\widehat{u}(\xi),\quad \widehat{P_{\leq
k}u}(\xi)=\chi_{\leq k}(|\xi|)\widehat{u}(\xi),\]
and $P_{\geq k}=I-P_{\leq k-1}$. Let $\cir^1$ be
the unit circle in $\R^2$. For $\ve e\in \cir^1$, define
$\widehat{P_{k,\ve e}u}(\xi)=\widetilde\chi_{k}(|\xi\cdot \ve
e|)\chi_k(|\xi|)\widehat{u}(\xi)$. Since for $|\xi|\sim 2^k$ we have
$\sum_{l=-5}^5\chi_{k+l}(\xi_1)+\sum_{l=-5}^5\chi_{k+l}(\xi_2)\sim
1$, then let
\[\beta_k^j(\xi)=\frac{\sum_{l=-5}^5\chi_{k+l}(\xi_j)}{\sum_{l=-5}^5\chi_{k+l}(\xi_1)+\sum_{l=-5}^5\chi_{k+l}(\xi_2)}\cdot \sum_{l=-1}^1\chi_{k+l}(|\xi|), \quad j=1,2.\]
Define the operator $\Theta_k^j$ on $L^2(\R^2)$ by
$\widehat{\Theta_k^jf}(\xi)=\beta_k^j(\xi)\hat f(\xi)$, $j=1,2$. Let $\ve
e_1=(1,0)$, $\ve e_2=(0,1)$. Then we have
\begin{align}\label{eq:Pkdec}
P_k=P_{k,\ve e_1}\Theta_k^1+P_{k,\ve e_2}\Theta_k^2.
\end{align}
For any $k\in \Z$, we define the modulation projectors $Q_k, Q_{\leq k}, Q_{\geq k}$ on $L^2(\R^2\times \R)$ by
\[\widehat{Q_ku}(\xi,\tau)=\chi_k(\tau+|\xi|^2)\widehat{u}(\xi,\tau),\quad \widehat{Q_{\leq
k}u}(\xi,\tau)=\chi_{\leq k}(\tau+|\xi|^2)\widehat{u}(\xi,\tau),\]
and $Q_{\geq k}=I-Q_{\leq k-1}$.

For any $\ve e\in \cir^1$, we can decompose $\R^2=\lambda \ve e
\oplus H_{\ve e}$, where $H_{\ve e}$ is the line with normal vector
$\ve e$, endowed with the induced measure. For $1\leq p, q< \infty$,
we define $L_{\ve e}^{p,q}$ the anisotropic Lebesgue space by
\[\norm{f}_{L_{\ve e}^{p,q}}=\brk{\int_\R\brk{\int_{H_{\ve e}\times \R} |f(\lambda \ve e+y,t)|^qdydt}^{p/q}d\lambda}^{1/p}\]
with the usual definition if $p=\infty$ or $q=\infty$. We write
$L_{\ve e_1}^{p,q}=L_{x_1}^pL_{x_2,t}^q$, $L_{\ve
e_2}^{p,q}=L_{x_2}^pL_{x_1,t}^q$.

We use $\theta\in \cir^1$ to denote the spherical variable. Let
$\Des$ be the Laplace operator on $\cir^1$, $\p_\theta$ be the
spherical derivative and $\Lambda_\theta=\sqrt{1-\Des}$. We identify
$\cir^1=\R/(2\pi \Z):=\T^1$. Denote by
$H^{s,p}_\theta=\Lambda_\theta^{-s}L^p$ the standard $L^p$ Sobolev
space on $\T^1$. We define $L_{\ve e}^{p,q}H_\theta^{s,r}$ by the
norm
\[\norm{f}_{L_{\ve e}^{p,q}H_\theta^{s,r}}=\normo{\norm{f(|x|\cos\theta ,|x|\sin \theta,t)}_{H_\theta^{s,r}}}_{L_{\ve e}^{p,q}}.\]
By the $SO(2)$ integration, we will also use the following form
\[\norm{f}_{L_{\ve e}^{p,q}H_\theta^{s,r}}=\normo{\brk{\int_0^{2\pi}|\Lambda_\theta ^s[f(A_\theta\cdot x, t)]|^rd\theta}^{1/r}}_{L_{\ve e}^{p,q}},\]
where $A_\theta=\begin{pmatrix} \cos\theta & -\sin \theta \\
\sin\theta & \cos\theta
\end{pmatrix}
$. For any function $f$, we denote the action $A_\beta
f(x)=f(A_\beta \cdot x)$. It's easy to see that $\p_\beta[A_\beta
f]=A_\beta (\p_\theta f)$. We use $\dot{B}^s_{p,q}$ to denote the
homogeneous Besov spaces on $\R^2$ which is the completion of the
Schwartz function under the norm
\[\norm{f}_{\dot{B}^s_{p,q}}=(\sum_{k\in \Z}2^{qsk}\norm{P_kf}_{L^p}^q)^{1/q}.\]
We define $\dot B^{s,\alpha}_{p,q,\theta}$ to be the space with the
norm $\norm{f}_{\dot
B_{p,q,\theta}^{s,\alpha}}=\norm{\Lambda_\theta^{\alpha}f}_{\dot
B_{p,q}^{s}}$. Then it's easy to see that $\norm{f}_{\dot
B_{p,q,\theta}^{s,1}}\sim \norm{f}_{\dot
B_{p,q}^{s}}+\norm{\p_\theta f}_{\dot B_{p,q}^{s}}$.

For the Schr\"odinger map, we need to use the Bourgain-type space associated to the Schr\"odinger equation.
In this paper we use the modulation-homogeneous version as in \cite{Bej2}. We define $X^{0,b,q}$ to be the completion of the space of Schwartz functions with the norm
\begin{align}\label{eq:Xbq}
\norm{f}_{X^{0,b,q}}=(\sum_{j\in \Z}2^{jbq}\norm{Q_j f}_{L^2_{t,x}}^q)^{1/2}.
\end{align}
If $q=2$ we simply write $X^{0,b}=X^{0,b,2}$. By the Plancherel's equality we have $\norm{f}_{X^{0,1}}=\norm{(i\p_t+\Delta)f}_{L^2_{t,x}}$. Since $X^{0,b,q}$  is not closed under conjugation, we also define the space $\bar X^{0,b,q}$ by the norm $\norm{f}_{\bar X^{0,b,q}}=\norm{\bar f}_{X^{0,b,q}}$, and similarly write $\bar X^{0,b}=\bar X^{0,b,2}$.
It's easy to see that $X^{0,b,q}$ function is unique modulo solutions of the homogeneous Schr\"odinger equation. For a more detailed description of the $X^{0,b,p}$ spaces we refer the readers to \cite{Tataru98} and \cite{Tao01}.

For any
space-time norm $X$, we define $XH^{s,p}_\theta$ by the norm
\[\norm{f}_{XH^{s,p}_\theta}=\norm{A_\theta f}_{XH^{s,p}_\theta}.\]
We conclude this section by a convolution property of the spherical
average space which implies that $P_k, \Theta_k^1,\Theta_k^2$ are bounded operators
in $XL^q_\theta$ if $X$ is space translation invariant.

\begin{lem}\label{lem:Pkbound}
Let $X$ be a space-time function space on $\R^2\times \R$ that is
space translation invariant. Then for $1\leq q\leq \infty$
\[\norm{f*g}_{XL_\theta^q}\leq C \norm{f}_{L_x^{1}L_\theta^{\infty}}\norm{g}_{XL_\theta^{q}}.\]
\end{lem}
\begin{proof}
We have
\begin{align*}
\norm{f*g}_{XL_\theta^q}\sim&\normo{\int
f(A_\beta x-y)g(y)dy}_{XL_\beta^q}\\
\sim&\normo{\int f(A_\beta x-A_\beta y)g(A_\beta
y)dy}_{XL_\beta^q}\les
\norm{f}_{L_x^{1}L_\theta^{\infty}}\norm{g}_{XL_\theta^{q}},
\end{align*}
where in the last inequality we used that $X$ is space translation
invariant.
\end{proof}

\section{Spherically averaged maximal function estimates}

In this section, we prove the spherically averaged maximal function
estimate. First, we consider the homogeneous case.

\begin{lem}\label{lem:max2d}
Assume $k\in \Z$. Then
\begin{align}
\norm{e^{it\Delta}P_kf}_{L_{\ve e}^{2,\infty}L^2_\theta}\les
2^{k/2}\norm{f}_2
\end{align}
\end{lem}

\begin{proof}

By the scaling invariance, we may assume $k=0$. Moreover, since
$e^{it\Delta}$ commutates with rotation, then by a rotation
transform we may assume $\ve e=(1,0)$. It reduces to prove
\begin{align}\label{eq:max2d1}
\norm{e^{it\Delta}P_0f}_{L_{x_1}^2L_{x_2,t}^{\infty}L^2_\theta}\les
\norm{f}_2
\end{align}
Using the H\"older inequality and Bernstein's inequality, we easily
get
\begin{align*}
\norm{e^{it\Delta}P_0 f}_{L^2_{|x_1|\leq 9}L_{x_2,t}^\infty
L_\theta^2} \leq \norm{e^{it\Delta}P_0 f}_{L_{x_1,x_2,t}^\infty}\leq
C \norm{f}_2.
\end{align*}
Then it remains to show
\begin{align}\label{eq:thmmainpf2}
\norm{e^{it\Delta}P_0 f}_{L^2_{|x_1|\geq 9}L_{x_2,t}^\infty
L_\theta^2} \leq C \norm{f}_2.
\end{align}
We will prove \eqref{eq:thmmainpf2} by two steps.

{\bf Step 1. } radial case

We assume $f$ is radial. It is well known that if $G(x)=g(|x|)$ is
radial and $G\in L^2(\R^n)$, then the Fourier transform of $G$ is
also radial (cf. \cite{Stein1}), and
\begin{align}\label{eq:FBform}
\hat{G} (\xi)= 2\pi  \int_{0}^{\infty} g(s) s^{n-1}(s
|\xi|)^{-\half{n-2}}J_{\frac{n-2}{2}}(s|\xi|)ds,
\end{align}
where $J_m(r)$ is the Bessel function
\begin{align*}
J_m(r)=\frac{(r/2)^m}{\Gamma(m+1/2)\pi^{1/2}}
\int_{-1}^1e^{irt}(1-t^2)^{m-1/2}dt, \ \ m>-1/2.
\end{align*}
Since $f$ is radial and denote $\widehat{f}(\xi)=h(|\xi|)$, then by
the formula \eqref{eq:FBform} we get $e^{it\Delta}P_0
f(x_1,x_2)=F_0(t,\sqrt{x_1^2+x_2^2})$, where for $\rho\geq 0$
\begin{align}\label{eq:radialfree}
F_0(t,\rho)= 2\pi  \int_{0}^{\infty}
e^{-its^2}\eta_0(s)h(s)sJ_{0}(s\rho)ds.
\end{align}
Therefore, to show \eqref{eq:thmmainpf2} it is equivalent to show
\begin{align}\label{eq:thmmainpf3}
\int_1^\infty \sup_{\rho\geq x,t\in \R}\aabs{F_0(t,\rho)}^2dx\leq C
\norm{h}_2^2.
\end{align}

To prove \eqref{eq:thmmainpf3}, we will use the decay properties at
the infinity of the Bessel function. More precisely, for $n\geq 2$
\begin{align}\label{eq:Besselinfty}
J_{\frac{n-2}{2}}(r)=c_n\frac{e^{ir}-e^{-ir}}{r^{1/2}}+c_nr^{\frac{n-2}{2}}e^{-ir}E_+(r)-c_nr^{\frac{n-2}{2}}e^{ir}E_-(r),
\end{align}
where $E_\pm(r)\les r^{-(n+1)/2}$ if $r\ge1$, see \cite{Stein2}.
Inserting \eqref{eq:Besselinfty} into \eqref{eq:radialfree}, we then
divide $F_0(t,|x|)$ into two parts: the main term and the error
term, namely
\begin{align}\label{eq:decfree}
F_0(t,\rho)=M(t,\rho)+E(t,\rho)
\end{align}
with
\begin{align*}
c{M}(t,\rho)=&\rho^{-\frac{1}{2}}\int_{\R}
\eta_0(s)h(s)s^{\frac{1}{2}} e^{i(\rho s-ts^2)}
ds+\rho^{-\frac{1}{2}}\int_{\R} \eta_0(s)
h(s)s^{\frac{1}{2}}e^{-i(\rho s+ts^2)} ds,\\
c{E}(t,\rho)=&\int_{\R} \eta_0(s)h(s)se^{-its^2- i\rho s}E_{+}(\rho
s)ds-\int_{\R} \eta_0(s)h(s)se^{-its^2+ i\rho s}E_{-}(\rho s)ds.
\end{align*}
For the error term, since $|E(t,\rho)|\les \rho^{-3/2}\norm{h}_2$,
then one get that
\begin{align*}
\int_1^\infty \sup_{\rho\geq x,t\in \R}\aabs{E(t,\rho)}^2dx\les
\int_1^\infty x^{-3}\norm{h}_2^2dx \les \norm{h}_2^2.
\end{align*}
It remains to bound the main term. From symmetry, it suffices to
show that
\begin{align}\label{eq:thmmainpf4}
\int_1^\infty \sup_{\rho\geq x,t\in
\R}\frac{1}{\rho}\aabs{\int_{0}^{\infty} \eta_0(s)h(s)e^{i(\rho
s-ts^2)}ds}^2dx\les \norm{h}_2^2.
\end{align}

Obviously,
\begin{align*}
&\int_1^\infty \sup_{\rho\geq x,t\in
\R}\frac{1}{\rho}\aabs{\int_{0}^{\infty} \eta_0(s)h(s)e^{i(\rho
s-ts^2)}ds}^2dx\\
&\les \sum_{k=0}^\infty 2^{-k}\int_1^\infty \sup_{2^k x\leq \rho\leq
2^{k+1}x,t\in \R}\frac{1}{x}\aabs{\int_{0}^{\infty}
\eta_0(s)h(s)e^{i(\rho s-ts^2)}ds}^2dx.
\end{align*}
Define the operator $T$ acting on $h\in L^2([1,3])$ as follows
\[T(h)(x)=\frac{1}{x^{1/2}}\int_{0}^{\infty} \eta_0(s)h(s)e^{i(x\rho
s-ts^2)}ds.\] Thus it suffices to show
\[\norm{Th}_{L^2_{x\in [1,\I]}L^\infty_{\rho\sim 2^k, t\in \R}}\leq C\norm{h}_2, \ \forall\, k\in \Z_+.\]
By $TT^*$ argument, it suffices to show
\begin{align}\label{eq:step1TT}
\norm{TT^*f}_{L^2_{x\in [1,\I]}L^\infty_{\rho\sim 2^k, t\in \R}}\leq
C\norm{f}_{L^2_{x\in [1,\I]}L^1_{\rho\sim 2^k, t\in \R}}.
\end{align}
Indeed, we have
\begin{align*}
TT^*f=\frac{1}{x^{1/2}}\int \brk{\int_{0}^{\infty}
\eta_0^2(s)e^{i((x\rho-x'\rho')
s-(t-t')s^2)}ds}\frac{1}{x'^{1/2}}f(x',\rho',t')dx'd\rho'dt'.
\end{align*}
By the stationary phase method, we have
\[\aabs{\int_{0}^{\infty}
\eta_0^2(s)e^{i((x\rho-x'\rho') s-(t-t')s^2)}ds}\les
(1+|x\rho-x'\rho'|)^{-1/2}.\] Thus, we get
\begin{align*}
|TT^*f|\les& \frac{1}{x^{1/2}}\int (1+|x\rho-x'\rho'|)^{-1/2}
\frac{1}{x'^{1/2}}|f(x',\rho',t')|dx'd\rho'dt'\\
\les& \int_{|x|\sim |x'|}
\frac{|f(x',\rho',t')|}{x^{1/2}x'^{1/2}}dx'd\rho'dt'+\int_{|x|\gg
|x'|}
\frac{|f(x',\rho',t')|}{2^{k/2}xx'^{1/2}}dx'd\rho'dt'\\
&+\int_{|x|\ll |x'|}
\frac{|f(x',\rho',t')|}{2^{k/2}x'x^{1/2}}dx'd\rho'dt'\\
:=&I+II+III.
\end{align*}
Now we show \eqref{eq:step1TT}. For the contribution of the term
$I$, we have
\[I\les M(\norm{f(\cdot,\rho,t)}_{L^1_{\rho,t}})(x)\]
where $M$ is the Hardy-Littlewood maximal operator. Then from the
$L^2$ boundedness of $M$, we see the estimate of $I$ is fine. The
estimate of $II,III$ simply follows from the H\"older inequality.

{\bf Step 2.} general case

We assume $f$ is nonradial. First, we make some reductions using the
spherical harmonics on $\cir^1$. For any function $f\in L^2(\R^2)$,
we can write
\[\widehat{f}(re^{i\theta})=\sum_{n\in \Z}f_n(r)e^{in\theta}.\]
Hence by the property of Fourier transform (see \cite{Stein1})
\[e^{it\Delta}P_0f(re^{i\theta})=\sum_{n\in \Z}2\pi i^{-n}T_n(f_n)(t,r)e^{in\theta},\]
where
\[T_n(f)(t,r)=\int e^{-it\rho^2}J_n(r\rho)\rho \chi_0(\rho)f(\rho)d\rho.\]
Thus \eqref{eq:thmmainpf2} becomes
\begin{align}\label{eq:step2pf1}
\norm{T_n(f_n)(t,|x|)}_{L_{x_1}^2L_{x_2,t}^\infty l_n^2}\les
\norm{f_n(|x|)}_{L_x^2l_n^2}.
\end{align}
To prove \eqref{eq:step2pf1}, it is equivalent to show
\begin{align}\label{eq:step2pf2}
\norm{T_n(f)(t,r)}_{L_{x}^2L_{r\geq |x|,t}^\infty}\les
\norm{f}_{L^2},
\end{align}
with a bound independent of $n\geq 0$

To prove \eqref{eq:step2pf2}, we need to use the uniform property of
$J_n$ with respect to $n$. We have
\begin{align*}
\norm{T_n(f)(t,r)}_{L_{x}^2L_{r\geq |x|,t}^\infty}\les&
\norm{T_n(f)(t,r)}_{L_{|x|\les n}^2L_{r\geq
|x|,t}^\infty}+\norm{T_n(f)(t,r)}_{L_{|x|\gg n}^2L_{r\geq
|x|,t}^\infty}\\
:=&A+B.
\end{align*}
First, we estimate the term $A$. By the Cauchy-Schwartz inequality,
we have
\[A\leq n^{1/2}\norm{T_n(f)(t,r)}_{L_{r\geq 0,t}^\infty}.\]
Thus it suffices to show $|T_n(f)(t,r)|\les n^{-1/2}$. If $r\gg n$
or $r\ll n$, this follows easily from the fact that $|J_n(r)|\les
n^{-1/2}$. It remains to show
\[\normo{\int e^{-it\rho^2}J_n(r\rho) \chi_0(\rho)f(\rho)d\rho}_{L^\infty_{r\sim n, t}}\les n^{-1/2}\norm{f}_2.\]
By $TT^*$ argument, it suffices to show
\[\normo{\int (\int e^{-i(t-t')\rho^2}J_n(r\rho)J_n(r'\rho) \chi_0(\rho)d\rho)g(t',r') dt'dr'}_{L^\infty_{r\sim n, t}}\les n^{-1}\norm{g}_{L^1_{r\sim n, t}}.\]
By the uniform decay of Bessel function (e.g. see Lemma 2.2 in
\cite{Guo2}),
\begin{align*}
|J_n(r)|\les (1+|r^2-n^2|)^{-1/4},
\end{align*}
it suffices to show
\[\sup_{r,r'\sim 1}|\int (1+|r^2\rho^2n^2-n^2|)^{-1/4}(1+|r'^2\rho^2n^2-n^2|)^{-1/4} \chi_0(\rho)d\rho|\les n^{-1},\]
which follows from the Cauchy-Schwartz inequality.

Now we estimate the term $B$. Since $r\geq |x|\gg n$, we have (given
in \cite{BC}, for the proof see Lemma 2.5 in \cite{Guo2})
\[J_n(r)=\frac{1}{\sqrt{2\pi}}\frac{e^{i\theta(r)}+e^{-i\theta(r)}}{(r^2-n^2)^{1/4}}+h(n,r):=J_n^1(r)+J_n^2(r)+J_n^3(r),\]
where
\[\theta(r)=(r^2-n^2)^{1/2}-\nu \arccos \frac{n}{r}-\frac \pi 4\]
and
\[|h(n,r)|\les r^{-1}.\]
Thus, we get
\begin{align*}
B\leq \sum_{j=1}^3\norm{T_n^j(f)(t,r)}_{L_{|x|\gg n}^2L_{r\geq
|x|,t}^\infty}:=\sum_{j=1}^3B_j
\end{align*}
where
\[T_n^j(f)(t,r)=\int e^{-it\rho^2}J^j_n(r\rho)\rho \chi_0(\rho)f(\rho)d\rho,\quad j=1,2,3.\]
For $B_3$, we use the decay of $h$ and get
\[\sup_{r\geq |x|}|T_n^3(f)(t,r)|\les \sup_{r\geq |x|}r^{-1}\norm{f}_2\les |x|^{-1}\norm{f}_2\]
which suffices to give the bound as desired. It remains to control
$B_1$ since the estimate for $B_2$ follows in the same way.

It suffices to show that
\begin{align*}
\int_{10n}^\infty \sup_{\rho\geq x,t\in \R}\aabs{\int_{0}^{\infty}
\frac{\chi_0(s)e^{i(\theta(s\rho)-ts^2)}}{(\rho^2s^2-n^2)^{1/4}}
h(s)ds}^2dx\les \norm{h}_2^2.
\end{align*}
Since $\rho\geq |x|\gg n$ and $s\sim 1$, then
$|(\rho^2s^2-n^2)^{-1/4}-(\rho s)^{-1/2}|\les |x|^{-5/2}n^2$, and
thus we get
\begin{align*}
&\int_{10n}^\infty \sup_{\rho\geq x,t\in \R}\aabs{\int_{0}^{\infty}
(\frac{1}{(\rho^2s^2-n^2)^{1/4}}-\frac{1}{\sqrt{\rho s}})
\chi_0(s)e^{i(\theta(s\rho)-ts^2)}h(s)ds}^2dx\\
\les& \int_{10n}^\infty x^{-5}n^4dx\cdot \norm{h}_2^2\les
\norm{h}_2^2.
\end{align*}
Therefore, it remains to show
\begin{align}
\int_{10n}^\infty \sup_{\rho\geq x,t\in
\R}\frac{1}{\rho}\aabs{\int_{0}^{\infty}
\chi_0(s)e^{i(\theta(s\rho)-ts^2)} h(s)ds}^2dx\les \norm{h}_2^2.
\end{align}
We proceed as in Step 1. Obviously,
\begin{align*}
&\int_{10n}^\infty \sup_{\rho\geq x,t\in
\R}\frac{1}{\rho}\aabs{\int_{0}^{\infty}
\chi_0(s)e^{i(\theta(s\rho)-ts^2)} h(s)ds}^2dx\\
&\les \sum_{k=0}^\infty 2^{-k}\int_{10n}^\infty \sup_{2^k x\leq
\rho\leq 2^{k+1}x,t\in \R}\frac{1}{x}\aabs{\int_{0}^{\infty}
\chi_0(s)e^{i(\theta(s\rho)-ts^2)} h(s)ds}^2dx.
\end{align*}
Define the operator $L$ acting on $h\in L^2([1,3])$ as follows
\[L(h)(x)=\frac{1}{x^{1/2}}\int_{0}^{\infty} \eta_0(s)h(s)e^{i(\theta(x\rho
s)-ts^2)}ds.\] Thus it suffices to show
\[\norm{Lh}_{L^2_{x\in [1,\I]}L^\infty_{\rho\sim 2^k, t\in \R}}\leq C\norm{h}_2, \ \forall\, k\in \Z_+.\]
By $TT^*$ argument, it suffices to show
\begin{align}\label{eq:step1TT}
\norm{LL^*f}_{L^2_{x\in [1,\I]}L^\infty_{\rho\sim 2^k, t\in \R}}\leq
C\norm{f}_{L^2_{x\in [1,\I]}L^1_{\rho\sim 2^k, t\in \R}}.
\end{align}
Indeed, we have
\begin{align*}
LL^*f=\frac{1}{x^{1/2}}\int \brk{\int_{0}^{\infty}
\chi_0^2(s)e^{i(\theta(x\rho s)-\theta(x'\rho'
s)-(t-t')s^2)}ds}\frac{1}{x'^{1/2}}f(x',\rho',t')dx'd\rho'dt'.
\end{align*}
Direct computation shows that for $r\gg n$
\begin{align*}
\theta'(r)=&(r^2-n^2)^{1/2}r^{-1}\sim 1,\\
\theta''(r)=&(r^2-n^2)^{-1/2}-(r^2-n^2)^{1/2}r^{-2}=(r^2-n^2)^{-1/2}n^2r^{-2}\les
r^{-1}.
\end{align*}
Thus by the stationary phase method, we have
\begin{align*}
\aabs{\int_{0}^{\infty}
\chi_0^2(s)e^{i(\theta(x\rho s)-\theta(x'\rho' s)-(t-t')s^2)}ds}\les
\begin{cases}
1, \quad {|x|\sim |x'|}\\
2^{-k/2}\max(|x|,|x'|)^{-1/2}, \, |x|\nsim |x'|.
\end{cases}
\end{align*}
With this the rest proof is the same as in step 1. We complete the
proof.
\end{proof}

Next, we derive the inhomogeneous estimate. Here we use an direct
argument of \cite{ywang} which is in the spirit of Lemma 7.5 in
\cite{BIKT}.

\begin{lem}\label{lem:max2din}
Let $k\in \Z$. Assume $u,F$ solves the equation
\[iu_t+\Delta u=F(x,t), \quad u(x,0)=0.\]
Then for any $\ve e\in \cir^1$ we have
\begin{align}
\norm{P_{k}u}_{L_{\ve e}^{2,\infty}L^2_\theta}\les \sup_{\ve e\in
\cir^1}\norm{F}_{L_{\ve e}^{1,2}}.
\end{align}
\end{lem}

\begin{proof}
By the scaling and rotational invariance, we may assume $k=0$ and
$\ve e=(1,0)$. $P_0u=U+V$ such that $\ft_xU$ is supported in
$\{|\xi|\sim 1: |\xi_1|\sim 1\}\times \R$ and $\ft_xV$ is supported
in $\{|\xi|\sim 1: |\xi_2|\sim 1\}\times \R$. Thus it suffices to
show
\begin{align}
\norm{U}_{L_{x_1}^2L_{x_2,t}^{\infty}L^2_\theta}\les
\norm{F}_{L_{x_1}^{1}L_{x_2,t}^{2}},\quad
\norm{V}_{L_{x_1}^2L_{x_2,t}^{\infty}L^2_\theta}\les
\norm{F}_{L_{x_2}^{1}L_{x_1,t}^{2}}.
\end{align}
We only show the estimate for $U$, since the estimate for $V$ is
identical. We still write $u=U$. We assume $\ft_xF$ is supported in
$\{|\xi|\sim 1: |\xi_1|\sim 1\}\times \R$. We have
\begin{align*}
u(t,x)=&\int_{\R^3}\frac{e^{it\tau}e^{ix\xi}}{\tau-|\xi|^2}\widehat{F}(\xi,\tau)d\xi
d\tau\\
=&\int_{\R^3}\frac{e^{it\tau}e^{ix\xi}}{\tau-|\xi|^2}\widehat{F}(\xi,\tau)(1_{\R^2\setminus\{(\tau,\xi_2):\tau-\xi_2^2\sim
1\}}+1_{\tau-\xi_2^2\sim 1})d\xi d\tau\\
:=&u_1+u_2.
\end{align*}
For $u_1$, we simply use the Plancherel equality and get
\[\norm{\Delta u_1}_{L^2}+\norm{\partial_t u_1}_2\leq \norm{F}_2,\]
and thus by Sobolev embedding and Bernstein's inequality we obtain
the desired estimate. Now we estimate $u_2$. Let
$G(x_1,\xi_2,\tau)=1_{|\tau-\xi_2^2|\sim 1}\ft_{x_2,t}F$. Then
\begin{align*}
u_2=&\int_\R \int_{\R^3}\frac{e^{it\tau}e^{ix\xi}}{\tau-|\xi|^2}[e^{-iy_1\xi_1}G(y_1,\xi_2,\tau)]d\xi d\tau dy_1\\
=&\int_\R T_{y_1}(G(y_1,\cdot))(t,x) dy_1
\end{align*}
where
\[T_{y_1}(f)(t,x)=\int_{\R^3}\frac{e^{it\tau}e^{ix\xi}}{\tau-|\xi|^2}[e^{-iy_1\xi_1}1_{|\xi_2|\les 1,\tau-\xi_2^2\sim
1}f(\xi_2,\tau)]d\xi d\tau.\] Thus it suffices to prove
\begin{align}
\norm{T_{y_1} (f)}_{L_{x_1}^2L_{x_2,t}^{\infty}L^2_\theta}\les
\norm{f}_{L^2},\quad \forall y_1\in \R.
\end{align}

Define $s=s(\tau,\xi_2)=\tau-\xi_2^2$, we have
\begin{align*}
T_{y_1}(f)(t,x)=&\int_{\R^2}1_{|\xi_2|\les 1,\tau-\xi_2^2\sim
1}\brk{\int\frac{e^{i(x_1-y_1)\xi_1}}{\tau-\xi_2^2-\xi_1^2}d\xi_1}e^{it\tau}e^{ix_2\xi_2}f(\xi_2,\tau)d\xi_2
d\tau\\
=&\int_{\R^2}\brk{\int
\frac{e^{i(x_1-y_1)\xi_1}}{2\sqrt{s}}(\frac{1}{\sqrt{s}+\xi_1}+\frac{1}{\sqrt{s}-\xi_1})d\xi_1}\\
&\quad \cdot e^{it\tau}e^{ix_2\xi_2}1_{|\xi_2|\les
1,\tau-\xi_2^2\sim
1}f(\xi_2,\tau)d\xi_2 d\tau\\
:=&I_1(f)+I_2(f).
\end{align*}
We only estimate $I_1$, since $I_2$ follows in the same way. By the
property of Hilbert transform, we get
\begin{align*}
I_1(f)(t,x)=&\int_{\R^2}
\frac{e^{-i(x_1-y_1)\sqrt{s}}}{2\sqrt{s}}i\sgn(x_1-y_1)\cdot
e^{it\tau}e^{ix_2\xi_2}1_{|\xi_2|\les 1,\tau-\xi_2^2\sim
1}f(\xi_2,\tau)d\xi_2 d\tau.
\end{align*}
Making a change of variable $\eta_1=-\sqrt{\tau-\xi_2^2}$,
$d\tau=2\eta_1 d\eta_1$, we obtain
\begin{align*}
I_1(f)(t,x)=&i\sgn(x_1-y_1)\int_{\R^2}
e^{it(\eta_1^2+\xi_2^2)}e^{i(x_1\eta_1+x_2\xi_2)}\\
&\cdot e^{-iy_1\eta_1}1_{|\xi_2|\les 1,\eta_1\sim
1}f(\xi_2,\eta_1^2+\xi_2^2)d\xi_2 d\eta_1.
\end{align*}
Thus, by Lemma \ref{lem:max2d} we get
\begin{align*}
\norm{I_1(f)}_{_{L_{x_1}^2L_{x_2,t}^{\infty}L^2_\theta}}\les
&\norm{1_{|\xi_2|\les 1}\cdot1_{\eta_1\sim
1}f(\xi_2,\eta_1^2+\xi_2^2)}_{L^2}\les \norm{f}_2.
\end{align*}
We complete the proof of the lemma.
\end{proof}

\section{Cubic Derivative NLS}

In this section we prove Theorem \ref{thm:2cDNLS}. The ideas is from
\cite{BIKT}. First, we define the main dyadic function space $F_k$
and $N_k$ for $k\in \Z$. If $f(x,t)\in L^2(\R^2\times \R)$ has
spatial frequency localized in $\{|\xi|\sim 2^k\}$, define
\begin{align*}
\norm{f}_{F_{k}}=&\norm{f}_{L_t^\infty
L_x^2}+\norm{f}_{L_t^4 L_x^4}+2^{k/6}\sup_{|j-k|\leq 20}\sup_{\ve{e}\in \cir^1}\norm{P_{j,\ve e} f}_{L_{\ve e}^{6,3}}\\
&+2^{-k/2}\sup_{\ve{e}\in
\cir^1}\norm{f}_{L_{\ve{e}}^{2,\infty}L_\theta^2}+2^{k/2}\sup_{|j-k|\leq
20}\sup_{\ve{e}\in
\cir^1}\norm{P_{j,\ve e} (A_\beta f)}_{L_{\ve{e}}^{\infty,2}L_\beta^2},\\
\norm{f}_{G_{k}}=&\norm{f}_{L_t^\infty L_x^2}+\norm{f}_{L_t^4
L_x^4}+2^{-k/2}\sup_{\ve{e}\in
\cir^1}\norm{f}_{L_{\ve{e}}^{2,\infty}L_\theta^2},\\
\norm{f}_{N_{k}}=&\inf_{f=f_1+f_2+f_3+f_4}(\norm{f_1}_{L_{t,x}^{4/3}}+2^{k/6}\norm{f_2}_{L_{\ve
e_1}^{3/2,6/5}}\\
&+2^{k/6}\norm{f_3}_{L_{\ve e_2}^{3/2,6/5}}+2^{-k/2}\sup_{\ve{e}\in
\cir^1}\norm{f_4}_{L_{\ve{e}}^{1,2}}).
\end{align*}
Then we define the space $F^{s},N^{s}$ with the following norm
\begin{align*}
\norm{u}_{F^{s}}=&\sum_{k\in \Z}2^{ks}(\norm{P_k
u}_{F_{k}}+\norm{P_k
\partial_\theta u}_{F_{k}}):=\sum_{k\in \Z}2^{ks}\norm{P_k
u}_{\widetilde F_{k}},\\
\norm{u}_{G^{s}}=&\sum_{k\in \Z}2^{ks}(\norm{P_k
u}_{G_{k}}+\norm{P_k
\partial_\theta u}_{G_{k}}):=\sum_{k\in \Z}2^{ks}\norm{P_k
u}_{\widetilde G_{k}},\\
\norm{u}_{N^{s}}=&\sum_{k\in \Z}2^{ks}(\norm{P_k
u}_{N_{k}}+\norm{P_k \partial_\theta u}_{N_{k}}):=\sum_{k\in
\Z}2^{ks}\norm{P_k u}_{\widetilde N_{k}}.
\end{align*}
Note that to use the spherically averaged maximal function estimate,
we need the spherically averaged local smoothing estimate.

\begin{lem}[Linear estimates] Assume $u,F,u_0$ solves the following
equation
\begin{align*}
i\partial_tu+\Delta u=&F,\quad u(0,x)=u_0.
\end{align*}
Then for any $s\in \R$, we have
\[\norm{u}_{F^{s}}=\norm{u_0}_{\dot B_{2,1,\theta}^{s,1}}+\norm{F}_{N^{s}}.\]
\end{lem}
\begin{proof}
By the definition, it suffices to show
\begin{align}
\norm{P_{k}u}_{F_{k}}\les \norm{P_ku_0}_2+ \norm{P_kF}_{N_{k}}.
\end{align}
Since $\Delta$ commutes with rotation and the local smoothing
estimate (see \cite{IK}), we have
\begin{align*}
\norm{P_{k,\ve e}A_\theta u}_{L_{\ve e}^{\infty,2}L^2_\theta}\les&
\norm{P_{k,\ve e}A_\theta u}_{L^2_\theta L_{\ve e}^{\infty,2}}\les
2^{-k/2}\norm{A_\theta u_0}_{L^2_\theta L^2}+2^{-k}\norm{\sup_{\ve
e\in \cir^1}\norm{A_\theta
F}_{L_{\ve e}^{1,2}}}_{L_\theta^2}\\
\les&2^{-k/2}\norm{u_0}_2+2^{-k}\sup_{\ve e\in
\cir^1}\norm{F}_{L_{\ve e}^{1,2}}.
\end{align*}
Similarly, in the above inequality we can replace $\sup_{\ve e\in
\cir^1}\norm{F}_{L_{\ve e}^{1,2}}$ by $\sup_{\ve e\in
\cir^1}\norm{F}_{L_{\ve e}^{3/2,6/5}}$ and
$\norm{F}_{L_{x,t}^{4/3}}$. The other components except for the
maximal function follow from the known linear estimate. For the
maximal function component, we use Lemma
\ref{lem:max2d}-\ref{lem:max2din} and the Christ-Kiselev lemma
\cite{CK} (or Lemma 7.3 in \cite{BIKT}).
\end{proof}

To prove Theorem \ref{thm:2cDNLS}, by the standard iteration method,
it suffices to show the trilinear estimates. We need the following
lemma.

\begin{lem}\label{lem:L2est}
Assume $k_1,k_2\in \Z$. Then
\begin{align*}
&\norm{P_{k_1}uP_{k_2}\bar
v}_{L_{t,x}^{2}}+\norm{P_{k_1}\partial_\theta uP_{k_2}\bar
v}_{L_{t,x}^{2}}+\norm{P_{k_1}uP_{k_2}\partial_\theta\bar
v}_{L_{t,x}^{2}}\\
&\les 2^{k_1/2}2^{-k_2/2}\norm{P_{k_1}u}_{\widetilde
G_{k_1}}\norm{P_{k_2}v}_{\widetilde F_{k_2}}.
\end{align*}
\end{lem}
\begin{proof}
Since $A_\beta$ commute with $P_k$, and by Lemma \ref{lem:Pkbound}
we have
\begin{align*}
&\norm{P_{k_1}uP_{k_2}\bar v}_{L_{t,x}^{2}}\\
=&\norm{P_{k_1}A_\beta
uP_{k_2}A_\beta \bar v}_{L_{t,x}^{2}L_\beta^2}\nonumber\\
\les& \norm{P_{k_1}A_\beta uP_{k_2,\ve e_1}\Theta_{k_2}^1A_\beta \bar
v}_{L_{t,x}^{2}L_\beta^2}+\norm{P_{k_1}A_\beta uP_{k_2,\ve
e_2}\Theta_{k_2}^2A_\beta \bar v}_{L_{t,x}^{2}L_\beta^2}\nonumber\\
\les& \norm{P_{k_1}A_\beta u}_{L_{\ve
e_1}^{2,\infty}L_\beta^\infty}\norm{P_{k_2,\ve e_1}A_\beta \bar
v}_{L_{\ve e_1}^{\infty,2}L_\beta^2}+\norm{P_{k_1}A_\beta u}_{L_{\ve
e_2}^{2,\infty}L_\beta^\infty}\norm{P_{k_2,\ve e_2}A_\beta \bar
v}_{L_{\ve e_2}^{\infty,2}L_\beta^2}\nonumber\\
\les&2^{k_1/2}2^{-k_2/2}\norm{P_{k_1}u}_{\widetilde
G_{k_1}}\norm{P_{k_2}v}_{\widetilde F_{k_2}}.
\end{align*}
For the other component, we have
\begin{align*}
&\norm{P_{k_1}\partial_\theta uP_{k_2}\bar
v}_{L_{t,x}^{2}}+\norm{P_{k_1}uP_{k_2}\partial_\theta\bar
v}_{L_{t,x}^{2}}\\
=& \norm{P_{k_1}\partial_\beta(A_\beta u)\cdot P_{k_2}A_\beta \bar
v]}_{L_{t,x}^{2}L_\beta^2}+\norm{P_{k_1}(A_\beta u)\cdot
P_{k_2}\partial_\beta(A_\beta \bar v)]}_{L_{t,x}^{2}L_\beta^2}\\
:=&I+II.
\end{align*}
For $I$, by the Sobolev embedding $H^1_{\theta}(\T^1)\hookrightarrow
L_\theta^\infty(\T^1)$ we have
\begin{align*}
I \les& \norm{P_{k_1}\partial_\beta(A_\beta u)\cdot P_{k_2,\ve
e_1}A_\beta \bar
v}_{L_{t,x}^{2}L_\beta^2}+\norm{P_{k_1}\partial_\beta(A_\beta
u)\cdot P_{k_2,\ve
e_2}A_\beta \bar v}_{L_{t,x}^{2}L_\beta^2}\nonumber\\
\les& \norm{P_{k_1}\partial_\beta(A_\beta u)}_{L_{\ve
e_1}^{2,\infty}L_\beta^2}\norm{P_{k_2,\ve e_1}A_\beta \bar
v}_{L_{\ve
e_1}^{\infty,2}L_\beta^\infty}\\
&+\norm{P_{k_1}\partial_\beta(A_\beta u)}_{L_{\ve
e_2}^{2,\infty}L_\beta^2}\norm{P_{k_2,\ve e_2}A_\beta \bar
v}_{L_{\ve e_2}^{\infty,2}L_\beta^\infty}\nonumber\\
\les&2^{k_1/2}2^{-k_2/2}\norm{P_{k_1}u}_{\widetilde
G_{k_1}}\norm{P_{k_2}v}_{\widetilde F_{k_2}}.
\end{align*}
For $II$, we have
\begin{align*}
II \les& \norm{P_{k_1}(A_\beta u)\cdot P_{k_2,\ve
e_1}\partial_\beta(A_\beta \bar
v)}_{L_{t,x}^{2}L_\beta^2}+\norm{P_{k_1}(A_\beta u)\cdot P_{k_2,\ve
e_2}\partial_\beta(A_\beta \bar v)}_{L_{t,x}^{2}L_\beta^2}\nonumber\\
\les& \norm{P_{k_1}(A_\beta u)}_{L_{\ve
e_1}^{2,\infty}L_\beta^\infty}\norm{P_{k_2,\ve
e_1}\partial_\beta(A_\beta \bar v)}_{L_{\ve
e_1}^{\infty,2}L_\beta^2}\\
&+\norm{P_{k_1}(A_\beta u)}_{L_{\ve
e_2}^{2,\infty}L_\beta^\infty}\norm{P_{k_2,\ve
e_2}\partial_\beta(A_\beta \bar
v)}_{L_{\ve e_2}^{\infty,2}L_\beta^2}\nonumber\\
\les&2^{k_1/2}2^{-k_2/2}\norm{P_{k_1}u}_{\widetilde
G_{k_1}}\norm{P_{k_2}v}_{\widetilde F_{k_2}}.
\end{align*}
We complete the proof of the lemma.
\end{proof}

\begin{lem}[Nonlinear estimates]\label{lem:tri}
Assume $i=1,2$, $s\geq 1/2$. Then
\begin{align*}
\norm{u\bar v \partial_{x_i}w}_{N^{s}}\les&
\norm{u}_{F^{s}}\norm{v}_{F^{1/2}}\norm{w}_{F^{1/2}}+\norm{u}_{F^{1/2}}\norm{v}_{F^{s}}\norm{w}_{F^{1/2}}+\norm{u}_{F^{1/2}}\norm{v}_{F^{1/2}}\norm{w}_{F^{s}}.
\end{align*}
\end{lem}
\begin{proof}
We only prove the case $s=1/2$, since the other case are similar. By
the definition, we have
\begin{align*}
&\norm{u\bar v
\partial_{x_i}w}_{N^{1/2}}\\
=&\sum_{k_4}2^{k_4/2}(\norm{P_{k_4}[u\bar
v \partial_{x_i}w]}_{N_{k_4}}+\norm{\partial_\theta P_{k_4}[u\bar
v \partial_{x_i}w]}_{N_{k_4}})\nonumber\\
\leq&\sum_{k_1,k_2,k_3,k_4}2^{k_4/2}(\norm{P_{k_4}[P_{k_1}uP_{k_2}\bar
v
\partial_{x_i}P_{k_3}w]}_{N_{k_4}}+\norm{\partial_\theta P_{k_4}[P_{k_1}uP_{k_2}\bar
v
\partial_{x_i}P_{k_3}w]}_{N_{k_4}})\\
:=&I+II.
\end{align*}
We will estimate the sum above case by case, according to the type
of frequency interactions. By symmetry, we may assume $k_1\leq k_2$.
We also assume that $k_2\leq k_3$, namely the derivative falls on
the largest frequency, since the other case $k_2> k_3$ can be
handled similarly.

{\bf Case 1:} $k_4\leq k_1+200$.

For this case, we use the Strichartz norm $L^{4/3}$ for
$N_{k,\alpha}$. By the properties of Fourier support of input
functions, we may assume $k_3\leq k_2+300$. Thus we have
\begin{align*}
I \les& \sum_{k_i: k_4\leq
\min(k_1,k_2,k_3)+5}2^{k_4/2}\norm{P_{k_4}[P_{k_1}uP_{k_2}\bar v
\partial_{x_i}P_{k_3}w]}_{L_{t,x}^{4/3}}\\
\les& \sum_{k_i: k_4\leq
k_1+5}2^{k_4/2}\norm{P_{k_1}u}_{L_{T,x}^{4}}2^{k_2/2}\norm{P_{k_2}v}_{L_{T,x}^{4}}2^{k_3/2}\norm{P_{k_3}w}_{L_{t,x}^{4}}\\
\les&\norm{u}_{F^{1/2}}\norm{v}_{F^{1/2}}\norm{w}_{F^{1/2}}.
\end{align*}
The estimate for $II$ is the same as $I$, since $\partial_\theta$
commutes with $P_k$.

{\bf Case 2:} $k_1+200<k_4\leq k_2+100$.

In this case we have $k_1<k_2-100$ and $k_3\leq k_2+200$. Then we
get
\begin{align*}
I\les&
\sum_{k_i}2^{k_4/2}2^{k_4/6}(\norm{P_{k_4}[P_{k_1}uP_{k_2}\bar v
\partial_{x_i}P_{k_3,\ve e_1}w]}_{L_{\ve e_1}^{3/2,6/5}}\\
&+\norm{P_{k_4}[P_{k_1}uP_{k_2}\bar v
\partial_{x_i}P_{k_3,\ve e_2}w]}_{L_{\ve e_2}^{3/2,6/5}}):=I_1+I_2.
\end{align*}
By symmetry we only estimate $I_1$. By Lemma \ref{lem:L2est} we get
\begin{align*}
I_1\les&
\sum_{k_i}2^{k_4/2}2^{k_4/6}2^{k_3}\norm{P_{k_1}uP_{k_2}\bar
v}_{L_{t,x}^{2}}\norm{P_{k_3,\ve e_1}w}_{L_{\ve e_1}^{6,3}}\\
\les&
\sum_{k_i}2^{k_1/2}2^{k_2/2}2^{k_3/2}\norm{P_{k_1}u}_{\widetilde
F_{k_1}}\norm{P_{k_2}v}_{\widetilde
F_{k_2}}\norm{P_{k_3}w}_{\widetilde
F_{k_2}}\\
\les&\norm{u}_{F^{1/2}}\norm{v}_{F^{1/2}}\norm{w}_{F^{1/2}}.
\end{align*}
For the term $II$, we have
\begin{align*}
II\les &\sum_{k_i}\norm{P_{k_4}[\partial_\theta(P_{k_1} uP_{k_2}\bar
v)\partial_{x_i}P_{k_3}w]}_{N_{k_4}}+\sum_{k_i}\norm{P_{k_4}[P_{k_1}
uP_{k_2}\bar v
\partial_\theta\partial_{x_i}P_{k_3}w]}_{N_{k_4}}\\
:=&II_1+II_2.
\end{align*}
For the term $II_1$, as for the term $I$ we get
\begin{align*}
II_1\les& \sum_{k_i}2^{k_4/2}2^{k_3}\norm{\partial_\theta(P_{k_1}
uP_{k_2}\bar v)}_{L_{t,x}^{2}}\norm{P_{k_3}w}_{\widetilde
F_{k_3}}\\
\les&\norm{u}_{F^{1/2}}\norm{v}_{F^{1/2}}\norm{w}_{F^{1/2}}.
\end{align*}
It remains to estimate the term $II_2$. Note that $[\partial_\theta,
\partial_{x_i}]=\partial_{x_i}$, so we get
\[II_2\les \sum_{k_i}\norm{P_{k_4}[P_{k_1}
uP_{k_2}\bar v
\partial_{x_i}P_{k_3}w]}_{N_{k_4}}+\sum_{k_i}\norm{P_{k_4}[P_{k_1}
uP_{k_2}\bar v
\partial_{x_i}P_{k_3}\partial_\theta w]}_{N_{k_4}}.\]
The first term on the righthand side above is just $I$, while the
second term can be handled exactly as for $I$.

{\bf Case 3:} $k_4>\max(k_1+200, k_2+100)$.

In this case we have $|k_4-k_3|\leq 5$. For the term $I$ by Lemma
\ref{lem:L2est} and noting that
$\norm{f}_{L_{\ve{e}}^{2,\infty}}\les
\norm{f}_{L_{\ve{e}}^{2,\infty}L_\theta^\infty}$, we get
\begin{align*}
I\les&
\sum_{k_i}2^{k_4/2}2^{-k_4/2}\norm{P_{k_4}[P_{k_1}uP_{k_2}\bar v
\partial_{x_i}P_{k_3}w]}_{L_{\ve{e}}^{1,2}}\\
\les& \sum_{k_i}\norm{P_{k_1}uP_{k_3}\partial_{x_i}w}_{L_{x,t}^{2}}\norm{P_{k_2}v}_{L_{\ve{e}}^{2,\infty}}\\
\les&\norm{u}_{F^{1/2}}\norm{v}_{F^{1/2}}\norm{w}_{F^{1/2}}.
\end{align*}
For the term $II$, we have
\begin{align*}
II \les& \sum_{k_i}\norm{P_{k_4}[P_{k_1}(\p_\theta u)P_{k_2}\bar v
\partial_{x_i}P_{k_3}w]}_{L_{\ve{e}}^{1,2}}+\sum_{k_i}\norm{P_{k_4}[P_{k_1} u(P_{k_2} \p_\theta
\bar v)
\partial_{x_i}P_{k_3}w]}_{L_{\ve{e}}^{1,2}}\\
&+\sum_{k_i}\norm{P_{k_4}[P_{k_1} u P_{k_2}\bar v
\p_\theta\partial_{x_i}P_{k_3}w]}_{L_{\ve{e}}^{1,2}}:=II_1+II_2+II_3.
\end{align*}
For the term $II_1$ we have
\begin{align*}
II_1\les& \sum_{k_i}\norm{P_{k_1}(\p_\theta
u)P_{k_3}\p_{x_i}w}_{L_{x,t}^{2}}\norm{P_{k_2}v}_{L_{\ve{e}}^{2,\infty}}\les\norm{u}_{F^{1/2,1}}\norm{v}_{F^{1/2,1}}\norm{w}_{F^{1/2,1}}.
\end{align*}
Similarly, we can bound the term $II_2$. For the term $II_3$, we use
the commutator as in Case 2 and then bound as $II_1$. Thus we finish
the proof.
\end{proof}

\section{Schr\"odinger map in two dimensions}

In this section, we prove Theorem \ref{thm:Schmap}. Consider the
Schr\"odinger maps
\begin{align}\label{eq:2dSchmap}
\p_t s=s\times \Delta_x s,\quad s(0)=s_0,
\end{align}
where $s:\R^2\times \R\to \cir^2\hookrightarrow \R^3$. Using the
stereographic projection
\[u=\frac{s_1+is_2}{1+s_3},\]
we see $u$ solves the equation
\begin{align}\label{eq:2dSchmap2}
i\partial_tu+\Delta u=\frac{2\bar
u}{1+|u|^2}\sum_{i=1}^n(\p_{x_i}u)^2.
\end{align}
Conversely, if $u$ solves \eqref{eq:2dSchmap2}, then
\[s=\bigg(\frac{2\Re u}{1+|u|^2},\frac{2\Im u}{1+|u|^2},\frac{1-|u|^2}{1+|u|^2}\bigg)\]
solves \eqref{eq:2dSchmap}. Now we focus on the study of \eqref{eq:2dSchmap2}.

We define
the main dyadic function space $Z_k$ and $Z_k$ for $k\in \Z$. If
$f(x,t)\in L^2(\R^2\times \R)$ has spatial frequency localized in
$\{|\xi|\sim 2^k\}$, define
\begin{align*}
\norm{f}_{Z_{k}}=&\norm{f}_{X^{0,1/2,\infty}}+2^{-k}\norm{f}_{X^{0,1}}+\norm{f}_{L_t^\infty
L_x^2}+\norm{f}_{L_t^4 L_x^4}+2^{k\e/2}\norm{f}_{L_t^4 L_x^{\frac{4}{1+\e}}L_\theta^3}\\
&+2^{-k/2}\sup_{\ve{e}\in
\cir^1}\norm{f}_{L_{\ve{e}}^{2,\infty}L_\theta^2}+2^{k/2}\sup_{|j-k|\leq
20}\sup_{\ve{e}\in
\cir^1}\norm{P_{j,\ve e} (A_\beta f)}_{L_{\ve{e}}^{\infty,2}L_\beta^2},\\
\norm{f}_{Y_{k}}=&\norm{f}_{L_t^\infty L_x^2}+\norm{f}_{L_t^4
L_x^4}+2^{k\e/2}\norm{f}_{L_t^4 L_x^{\frac{4}{1+\e}}L_\theta^3}+2^{-k/2}\sup_{\ve{e}\in
\cir^1}\norm{f}_{L_{\ve{e}}^{2,\infty}L_\theta^2}\\
&+2^{-k}\inf_{f=f_1+f_2}(\norm{f_1}_{X^{0,1}}+\norm{f_2}_{\bar X^{0,1}}),\\
\norm{f}_{W_{k}}=&\inf_{f=f_1+f_2+f_3+f_4}(\norm{f_1}_{L_{t,x}^{4/3}}
+\norm{f_2}_{L_t^1L_x^2}+2^{-k/2}\sup_{\ve{e}\in
\cir^1}\norm{f_3}_{L_{\ve{e}}^{1,2}}+\norm{f_4}_{X^{0,-1/2,1}})\\
&+2^{-k}\norm{f}_{L^2_{t,x}},
\end{align*}
where $0<\e\ll 1$ will be a fixed universal number (e.g. $\e<0.01$ would work).
Then we define the space $Z^{s},W^{s}$ with the following norm
\begin{align*}
\norm{u}_{Z^{s}}=&\sum_{k\in \Z}2^{ks}(\norm{P_k u}_{Z_{k}}+\norm{P_k
\partial_\theta u}_{Z_{k}}):=\sum_{k\in \Z}2^{ks}\norm{P_k
u}_{\widetilde Z_{k}},\\
\norm{u}_{Y^{s}}=&\sum_{k\in \Z}2^{ks}(\norm{P_k
u}_{Y_{k}}+\norm{P_k
\partial_\theta u}_{Y_{k}}):=\sum_{k\in \Z}2^{ks}\norm{P_k
u}_{\widetilde Y_{k}},\\
\norm{u}_{W^{s}}=&\sum_{k\in \Z}2^{ks}(\norm{P_k
u}_{W_{k}}+\norm{P_k
\partial_\theta u}_{W_{k}}):=\sum_{k\in \Z}2^{ks}\norm{P_k
u}_{\widetilde W_{k}}.
\end{align*}
To prove Theorem \ref{thm:Schmap}, it
suffices to prove
\begin{thm}\label{thm:2cDNLS2}
Assume $n=2$, $u_0\in \dot B_{2,1,\theta}^{1,1}$ with
$\norm{u_0}_{\dot B_{2,1,\theta}^{1,1}}=\e_0\ll 1$. Then there
exists a unique global solution $u$ to \eqref{eq:2dSchmap2} such
that $\norm{u}_{Z^{1}}\les \e_0$. Moreover, the map $u_0\to u$ is
Lipshitz from $\dot B_{2,1,\theta}^{1,1}$ to $C(\R;\dot
B_{2,1,\theta}^{1,1})$, and scattering holds in this space.
\end{thm}

We will prove the above theorem via picard iteration argument. We need to prove some linear estimates and nonlinear estimates.

\begin{lem}[Linear estimates]
Assume $u,f,u_0$ solves the following equation
\[(i\p_t+\Delta)u=f, \quad u(0)=u_0.\]
Then we have
\begin{align}
\norm{u}_{Z_k}\les \norm{u_0}_{L_x^2}+\norm{f}_{W_k}.
\end{align}
\end{lem}
\begin{proof}
Most of the estimates were given in \cite{BIKT}.
We only need to deal with the component $2^{k\e/2}\norm{f}_{L_t^4L_x^{\frac{4}{1+\e}}L_\theta^3}$. We will use the generalised Strichartz estimates proved in \cite{Guo2}. In \cite{Guo2} the author proved the following estimate
\begin{align*}
\norm{e^{it\Delta}P_0u_0}_{L_t^2L_x^{6+}L_\theta^2}\les \norm{u_0}_{L^2(\R^2)}.
\end{align*}
Interpolating the above with the following estimate proved in \cite{KaOz} (Theorem 3.1):
\begin{align*}
\norm{e^{it\Delta}P_0u_0}_{L_t^2L_x^{\infty}L_\theta^p}\les \norm{u_0}_{L^2(\R^2)}, \quad 1\leq p<\infty,
\end{align*}
we get
\begin{align*}
\norm{e^{it\Delta}P_0u_0}_{L_t^2L_x^{\frac{2}{\e}}L_\theta^6}\les \norm{u_0}_{L^2(\R^2)}.
\end{align*}
Interpolating above with the trivial estimate $\norm{e^{it\Delta}P_0u_0}_{L_t^\infty L_x^{2}L_\theta^2}\les \norm{u_0}_{L^2(\R^2)}$, we get
\begin{align*}
\norm{e^{it\Delta}P_0u_0}_{L_t^4L_x^{\frac{4}{1+\e}}L_\theta^3}\les \norm{u_0}_{L^2(\R^2)}.
\end{align*}
Then by scaling transform we complete the proof.
\end{proof}

We use Taylor's expansion to rewrite the nonlinear term: if
$\norm{u}_\infty<1$
\[\frac{2\bar
u}{1+|u|^2}\sum_{i=1}^n(\p_{x_i}u)^2=\sum_{k=0}^\infty 2\bar u
(-|u|^2)^k\sum_{i=1}^n(\p_{x_i}u)^2.\] We prove

\begin{lem} [Nonlinear estimates]
The following estimates holds
\[\norm{\bar u (-|u|^2)^{k}\sum_{i=1}^n(\p_{x_i}u)^2}_{W^1}\les C^{2k} \norm{u}_{Y^1}^{2k+1}\norm{u}_{Z^1}\norm{u}_{Z^1}.\]
\end{lem}

The above lemma will follow from the following lemmas.

\begin{lem}
(1) If $j\geq 2k-100$ and $X$ is a space-time translation invariant
Banach space, then $Q_{\leq j}P_k$ is bounded on $X$ with bound
independent of $j,k$.

(2) For any $j,k$, $Q_{\leq j}P_{k,\ve e}$ is bounded on $L_{\ve
e}^{p, 2}$ and $Q_{\leq j}$ is bounded on $L_t^pL_x^2$ for $1\leq p
\leq \infty$, with bound independent of $j,k$.
\end{lem}

\begin{proof}
(1) The operator $Q_{\leq j}P_k$ corresponds to the space-time
multiplier with symbol
$\eta(\frac{\tau+|\xi|^2}{2^j})\chi(\frac{|\xi|}{2^k})$. It suffices
to show
\[\norm{\F^{-1}\eta(\frac{\tau+|\xi|^2}{2^j})\chi(\frac{|\xi|}{2^k})}_{L^1_{t,x}}\les 1, \quad \forall j,k.\]
From integration by part and the condition $j\geq 2k-100$, we get
\begin{align*}
|\int_{\R^n\times
\R}\eta(\frac{\tau+|\xi|^2}{2^j})\chi(\frac{|\xi|}{2^k})e^{ix\xi}e^{it\tau}d\xi
d\tau|\les 2^{j}(1+2^j|t|)^{-2}\cdot 2^{k}(1+2^k|x|)^{-n-1},
\end{align*}
which completes the proof.

(2) This was proved in \cite{Bej2}. In view of part (1), we may assume
$j\leq 2k-100$. By rotation, we may take $\ve e=\ve e_1$. By
Plancherel's equality, it suffices to prove
\begin{align*}
\normo{\int_{\R}e^{ix_1\xi_1}\eta(\frac{\tau+|\xi|^2}{2^j})\chi(\frac{|\xi|}{2^k})\chi(\frac{|\xi_1|}{2^k})
(\F_{x_1}f)(\xi,\tau) d\xi_1}_{L_{x_1}^p L_{\tau,\bar \xi}^2}\les
\norm{f}_{L_{x_1}^p L_{\tau,\bar \xi}^2}.
\end{align*}
Furthermore, it suffices to show
\begin{align}
\sup_{\bar
\xi,\tau}\normo{\int_{\R}e^{ix_1\xi_1}\eta(\frac{\tau+|\xi|^2}{2^j})\chi(\frac{|\xi|}{2^k})\chi(\frac{|\xi_1|}{2^k})
d\xi_1}_{L_{x_1}^1}\les 1.
\end{align}
For fixed $\bar \xi,\tau$, since $|\tau+|\bar \xi|^2+\xi_1^2|\les
2^j\ll |\xi_1|^2$, thus $\tau+|\bar \xi|^2$ is negative, and we have
either $|\xi_1-\sqrt{-\tau-|\bar \xi|^2}|\les 2^{j-k}$ or
$|\xi_1+\sqrt{-\tau-|\bar \xi|^2}|\les 2^{j-k}$. Thus $\xi_1$ varies
in a ball of size $2^{j-k}$. From integration by part and the
condition $j\leq 2k-100$, we get
\begin{align*}
|\int_{\R}e^{ix_1\xi_1}\eta(\frac{\tau+|\xi|^2}{2^j})\chi(\frac{|\xi|}{2^k})\chi(\frac{|\xi_1|}{2^k})
d\xi_1|\les 2^{j-k}(1+2^{j-k}|x_1|)^{-2}.
\end{align*}

To see $Q_{\leq j}$ is bounded on $L_t^pL_x^2$, by Plancherel's
equality it suffices to prove
\begin{align*}
\normo{\int
e^{it\tau}\eta(\frac{\tau-\xi^2}{2^j})\F_tf(\tau,\xi)d\tau}_{L_t^pL_{\xi}^2}\les
\norm{f}_{L_t^pL_\xi^2},
\end{align*}
which is equivalent to show
\begin{align*}
\normo{\int
e^{it\tau}\eta(\frac{\tau+\xi^2}{2^j})(\F_tf)(\tau+|\xi|^2,\xi)d\tau}_{L_t^pL_{\xi}^2}\les
\norm{f}_{L_t^pL_\xi^2}.
\end{align*}
The above inequality is trivial.
\end{proof}

\begin{lem}\label{lem:L2est2}
Assume $k_1,k_2,k_3\in \Z$. Then
\begin{align*}
&\norm{P_{k_3}(P_{k_1}uP_{k_2}v)}_{L_{t,x}^{2}}+\norm{P_{k_3}(P_{k_1}\p_\theta uP_{k_2}v)}_{L_{t,x}^{2}}+
\norm{P_{k_3}(P_{k_1}uP_{k_2}\p_\theta v)}_{L_{t,x}^{2}}\\
\les&
2^{\e[\min(k_1,k_2,k_3)-\max(k_1,k_2,k_3)]/2}\norm{P_{k_1}u}_{\widetilde Y_{k_1}}\norm{P_{k_2}v}_{\widetilde Y_{k_2}}.
\end{align*}
\end{lem}
\begin{proof}
We only estimate $\norm{P_{k_3}(P_{k_1}\p_\theta uP_{k_2}v)}_{L_{t,x}^{2}}$.
If $k_3\leq \min(k_1,k_2)$, then
\begin{align*}
\norm{P_{k_3}(P_{k_1}\p_\theta uP_{k_2}v)}_{L_{t,x}^{2}}\les&2^{\e k_3}\norm{P_{k_1}\p_\theta uP_{k_2}v}_{L_{t}^{2}L_x^{\frac{2}{1+\e}}}\\
\les&2^{\e k_3}2^{-\e (k_1+k_2)/2}\norm{P_{k_1}\p_\theta u}_{L_t^4L_x^{\frac{4}{1+\e}}L_\theta^3}\norm{P_{k_2}v}_{L_t^4L_x^{\frac{4}{1+\e}}L_\theta^\infty}\\
\les&2^{\e k_3}2^{-\e (k_1+k_2)/2}\norm{P_{k_1}u}_{\widetilde Y_{k_1}}\norm{P_{k_2}v}_{\widetilde Y_{k_2}}.
\end{align*}
If $k_1\leq \min(k_2,k_3)$, then
\begin{align*}
\norm{P_{k_3}(P_{k_1}\p_\theta uP_{k_2}v)}_{L_{t,x}^{2}}
\les&\norm{P_{k_1}\p_\theta u}_{L_{t}^{4}L_x^{\frac{4}{1-\e}}L_\theta^3}\norm{P_{k_2}v}_{L_{t}^4 L_x^{\frac{4}{1+\e}}L_\theta^\infty}\\
\les&2^{\e (k_1-k2)/2}\norm{P_{k_1}u}_{\widetilde Y_{k_1}}\norm{P_{k_2}v}_{\widetilde Y_{k_2}}.
\end{align*}
If $k_2\leq \min(k_1,k_3)$, the proof is identical to the above case. 
\end{proof}

\begin{lem}[Algebra properties]
If $s\geq 1$, then we have
\begin{align*}
\norm{uv}_{Y^{s}}\les&
\norm{u}_{Y^{s}}\norm{v}_{Y^{1}}+\norm{u}_{Y^{1}}\norm{v}_{Y^{s}}.
\end{align*}
\end{lem}
\begin{proof}
By the definition we have $\norm{u}_{L^\infty_{x,t}}\leq
\norm{u}_{L^\infty_{t}\dot B_{2,1}^1}+\norm{\p_\theta
u}_{L^\infty_{t}\dot B_{2,1}^1}\les \norm{u}_{Y^{1}}$. The Lebesgue component can be easily handled by
para-product decomposition and H\"older's inequality. Now we deal with $X^{s,b}$-type space.  It suffices to show
\begin{align}\label{eq:algpf1}
\sum_k \norm{\p_\theta^jP_k(fg)}_{X^{0,1}+\bar X^{0,1}}\les \norm{f}_{Y^1}\norm{g}_{Y^1}, \quad j=0,1.
\end{align}
For simplicity of notations,we write $X=X^{0,1}, \bar X=\bar X^{0,1}$. First we consider $j=0$. The left-hand side of the above inequality is bounded by
\begin{align*}
\sum_{k_3} \norm{P_{k_3}(fg)}_{X+\bar X}\les& \sum_{k_i}\norm{P_k(P_{k_1}fP_{k_2}g)}_{X+\bar X}\\
\leq& (\sum_{k_i:k_1\leq k_2}+\sum_{k_i:k_1> k_2})\norm{P_{k_3}(P_{k_1}fP_{k_2}g)}_{X+\bar X}\\
:=&I+II.
\end{align*}
By symmetry, we only estimate the term $I$. Assume $P_{k_1}f=P_{k_1}f_1+P_{k_1}f_2$, $P_{k_2}g=P_{k_2}g_1+P_{k_2}g_2$ such that
\[\norm{P_{k_1}f_1}_{X}+\norm{P_{k_1}f_2}_{\bar X}\les \norm{P_{k_1}f}_{X+\bar X}, \quad \norm{P_{k_2}g_1}_{X}+\norm{P_{k_2}g_2}_{\bar X}\les \norm{P_{k_2}g}_{X+\bar X}.\]
Then we have
\begin{align*}
I\les& \sum_{k_i:k_1\leq k_2}\sum_{j=1}^2\norm{P_{k_3}(P_{k_1}f_jP_{k_2}g_1)}_{X}+\sum_{k_i:k_1\leq k_2}\sum_{j=1}^2\norm{P_{k_3}(P_{k_1}f_jP_{k_2}g_2)}_{\bar X}\\
:=&I_1+I_2.
\end{align*}

We only estimate the term $I_1$ since the other term $I_2$ can be estimated in a similar way.  First we assume $k_3\leq k_1+5$. We have
\begin{align*}
I_1\les& \sum_{k_i:k_1\leq k_2}\sum_{j=1}^2(\norm{P_{k_3}Q_{\leq k_1+k_2+9}(P_{k_1}f_jP_{k_2}g_1)}_{X}+\norm{P_{k_3}Q_{\geq k_1+k_2+10}(P_{k_1}f_jP_{k_2}g_1)}_{X})\\
:=&I_{11}+I_{12}.
\end{align*}
For the term $I_{11}$, by Lemma \ref{lem:L2est2} we get
\begin{align*}
I_{11}\les&\sum_{k_i:k_1\leq k_2}\sum_{j=1}^22^{k_1+k_2}\norm{P_{k_3}(P_{k_1}f_jP_{k_2}g_1)}_{L^2_{t,x}}\\
\les& \sum_{k_i:k_1\leq k_2}\sum_{j=1}^22^{k_1+k_2}2^{\e(k_3-k_1)}\norm{P_{k_1}f_j}_{\wt Y_{k_1}}\norm{P_{k_2}g_1}_{\wt Y_{k_2}}\\
\les&  \norm{f}_{Y^1}\norm{g}_{Y^1}.
\end{align*}
For the term $I_{12}$, we need to exploit the nonlinear interactions.  We have
\begin{align*}
&\ft P_{k_3}Q_{\geq k_1+k_2+10}(P_{k_1}f_jP_{k_2}g_1)\\
&=\chi_{k_3}(\xi_3)\chi_{\geq k_1+k_2+10}(\tau_3+|\xi_3|^2)
\int_{\xi_3=\xi_1+\xi_2,\tau_3=\tau_1+\tau_2}\chi_{k_1}(\xi_1)\widehat{f_j}(\tau_1,\xi_1)\chi_{k_2}(\xi_2)\widehat{g_1}(\tau_2,\xi_2).
\end{align*}
We assume $j=1$ since $j=2$ is similar.  On the plane $\{\xi_3=\xi_1+\xi_2,\tau_3=\tau_1+\tau_2\}$ we have
\begin{align}
\tau_3+|\xi_3|^2=\tau_1+|\xi_1|^2+\tau_2+|\xi_2|^2-H(\xi_1,\xi_2)
\end{align}
where $H$ is the resonance function in the product
$P_{k_3}(P_{k_1}f_jP_{k_2}g_1)$
\begin{align}
H(\xi_1,\xi_2)=|\xi_1|^2+|\xi_2|^2-|\xi_1+\xi_2|^2.
\end{align}
Since $|H|\les 2^{k_1+k_2}$, then
one of $P_{k_1}f_j$, $P_{k_2}g_1$ has modulation larger than the
output modulation, namely
\[\max(|\tau_1+|\xi_1|^2|,|\tau_2+|\xi_2|^2|)\ges |\tau_3+|\xi_3|^2|.\]
If $P_{k_1}f_j$ has larger modulation, then
\begin{align*}
I_{12}\les& \sum_{k_i:k_1\leq k_2}\normo{2^{j_3}\norm{P_{k_3}Q_{j_3}(P_{k_1}f_jP_{k_2}g_1)}_{L^2_{t,x}}}_{l^2_{j_3\geq k_1+k_2}}\\
\les&
\sum_{k_i:k_1\leq k_2}\sum_{j=1}^22^{k_3}(\sum_{j_3\geq k_1+k_2}2^{2j_3}\norm{Q_{\geq j_3}P_{k_1}f_j}^2_{L_{t,x}^2})^{1/2}\norm{P_{k_2}g_1}_{L^\infty_tL_x^2}\\
\les&\sum_{k_i:k_1\leq
k_2}\sum_{j=1}^22^{k_3}2^{k_1}\norm{P_{k_1}f_j}_{X+\bar X}\norm{P_{k_2}g_1}_{Y_{k_2}}\les
\norm{f}_{Y^{1}}\norm{g}_{Y^{1}}.
\end{align*}
If $P_{k_1}g_1$ has larger modulation, then
\begin{align*}
I_{12}\les& \sum_{k_i:k_1\leq k_2}\sum_{j=1}^22^{k_3}\norm{P_{k_1}f_j}_{L_{t}^\infty L_x^2}(\sum_{j_3\geq k_1+k_2}2^{2j_3}\norm{P_{k_2}Q_{\geq j_3}g_1}^2_{L^2_tL_x^2})^{1/2}\\
\les&\sum_{k_i:k_1\leq
k_2}\sum_{j=1}^22^{k_3}2^{k_2}\norm{P_{k_1}f_j}_{Y_{k_1}}\norm{P_{k_2}g_1}_{X}\les
\norm{f}_{Y^{1}}\norm{g}_{Y^{1}}.
\end{align*}

Now we assume $k_3\geq k_1+6$. In this case we have $|k_2-k_3|\leq 4$.
\begin{align*}
I_1\les& \sum_{k_i:k_1\leq k_2}\sum_{j=1}^2(\norm{P_{k_3}Q_{\leq k_1+k_2+9}(P_{k_1}f_jP_{k_2}g_1)}_{X}+\norm{P_{k_3}Q_{\geq k_1+k_2+10}(P_{k_1}f_jP_{k_2}g_1)}_{X})\\
:=&\tilde I_{11}+\tilde I_{12}.
\end{align*}
By Lemma \ref{lem:L2est2} we get
\begin{align*}
\tilde I_{11}\les& \sum_{k_i:k_1\leq k_2}\sum_{j=1}^2(\norm{P_{k_3}Q_{\leq k_1+k_2+9}(P_{k_1}f_jP_{k_2}g_1)}_{X}\\
\les&\sum_{k_i:k_1\leq k_2}2^{k_1+k_2}\norm{P_{k_1}f_j}_{L_{t,x}^4}\norm{P_{k_2}g_1}_{L_{t,x}^4}\les \norm{f}_{Y^1}\norm{g}_{Y^1}.
\end{align*}
For the term $\tilde I_{12}$, similarly as the term $I_{12}$, we may assume $P_{k_1}f_j$ (or $P_{k_2}g_1$) has modulation $\ges 2^{k_1+k_2}$.
If $P_{k_1}f_j$ has larger modulation, then
\begin{align*}
I_{12}\les& \sum_{k_i:k_1\leq k_2}\normo{2^{j_3}\norm{P_{k_3}Q_{j_3}(P_{k_1}f_jP_{k_2}g_1)}_{L^2_{t,x}}}_{l^2_{j_3\geq k_1+k_2}}\\
\les&
\sum_{k_i:k_1\leq k_2}\sum_{j=1}^22^{k_1}(\sum_{j_3\geq k_1+k_2}2^{2j_3}\norm{Q_{\geq j_3}P_{k_1}f_j}^2_{L_{t,x}^2})^{1/2}\norm{P_{k_2}g_1}_{L^\infty_tL_x^2}\\
\les&\sum_{k_i:k_1\leq
k_2}\sum_{j=1}^22^{k_1}2^{k_1}\norm{P_{k_1}f_j}_{X+\bar X}\norm{P_{k_2}g_1}_{Y_{k_2}}\les
\norm{f}_{Y^{1}}\norm{g}_{Y^{1}}.
\end{align*}
If $P_{k_1}g_1$ has larger modulation, then
\begin{align*}
I_{12}\les& \sum_{k_i:k_1\leq k_2}\sum_{j=1}^22^{k_1}\norm{P_{k_1}f_j}_{L_{t}^\infty L_x^2}(\sum_{j_3\geq k_1+k_2}2^{2j_3}\norm{P_{k_2}Q_{\geq j_3}g_1}^2_{L^2_tL_x^2})^{1/2}\\
\les&\sum_{k_i:k_1\leq
k_2}\sum_{j=1}^22^{k_1}2^{k_2}\norm{P_{k_1}f_j}_{Y_{k_1}}\norm{P_{k_2}g_1}_{X}\les
\norm{f}_{Y^{1}}\norm{g}_{Y^{1}}.
\end{align*}
For $j=1$ in \eqref{eq:algpf1}, we see the above argument can be easily modified. Thus, we complete the proof.
\end{proof}

\begin{lem} \label{lem:nonL2est}
We have
\begin{align}
&\sum_{k_1,k_2,k_3}(\norm{P_{k_3}[u\sum_{i=1}^2 (\p_{x_i}P_{k_1}v
\p_{x_i}P_{k_2}w)]}_{L_{t,x}^2}+\norm{\p_\theta P_{k_3}[u\sum_{i=1}^2 (\p_{x_i}P_{k_1}v
\p_{x_i}P_{k_2}w)]}_{L_{t,x}^2})\nonumber\\
&\les
\norm{u}_{Y^1}\norm{v}_{Z^1}\norm{w}_{Z^1}.\label{eq:nonL2est}
\end{align}
\end{lem}
\begin{proof}
We only estimates the first term on the left-hand side, since the other term is similar.
We have
\begin{align*}
\mbox{First term on LHS of \eqref{eq:nonL2est}}\les&
\sum_{k_1,k_2,k_3}\norm{P_{k_3}[P_{\geq k_3-10}u\sum_{i=1}^n
(\p_{x_i}P_{k_1}v \p_{x_i}P_{k_2}w)]}_{L_{t,x}^2}\\
&+ \sum_{k_1,k_2,k_3}\norm{P_{k_3}[P_{\leq k_3-10}u\sum_{i=1}^n
(\p_{x_i}P_{k_1}v \p_{x_i}P_{k_2}w)]}_{L_{t,x}^2}\\
:=&I+II.
\end{align*}
For the term $I$, by Lemma \ref{lem:L2est} we get
\begin{align*}
I\les& \sum_{k_1,k_2,k_3}2^{k_3}\norm{P_{\geq k_3-10}u}_{L_t^\infty
L_x^2}\norm{\sum_{i=1}^n (\p_{x_i}P_{k_1}v
\p_{x_i}P_{k_2}w)]}_{L_{t,x}^2}\\
\les&\sum_{k_1,k_2,k_3}2^{k_3}2^{k_1+k_2}\norm{P_{\geq
k_3-10}u}_{L_t^\infty L_x^2}\norm{P_{k_1}v}_{F_{k_1}}\norm{P_{k_2}w}_{F_{k_2}}\\
\les& \norm{u}_{Y^1}\norm{v}_{Z^1}\norm{w}_{Z^1}.
\end{align*}
For the term $II$ we have
\begin{align*}
II\les& \norm{u}_{Y^1} \sum_{k_1,k_2,k_3}\norm{\tilde
P_{k_3}\sum_{i=1}^n (\p_{x_i}P_{k_1}v
\p_{x_i}P_{k_2}w)]}_{L_{t,x}^2}
\end{align*}
We may assume $k_1\leq k_2$ in the above summation. If $k_3\geq
k_2-9$, then $|k_3-k_2|\leq 5$ and thus
\begin{align*}
II\les& \norm{u}_{Y^1} \sum_{k_1\leq k_2,|k_2-k_3|\leq
5}2^{k_1+k_2}2^{(k_1-k_2)/2}\norm{P_{k_1}v}_{F_{k_1}}\norm{P_{k_2}w}_{F_{k_2}}\\
\les& \norm{u}_{Y^1}\norm{v}_{Z^1}\norm{w}_{Z^1}.
\end{align*}
If $k_3\leq k_2-10$, then $|k_1-k_2|\leq 5$.  Thus we get
\begin{align*}
II\les& \norm{u}_{Y^1} \sum_{k_1,k_2,k_3}2^{k_3\e}\norm{\tilde
P_{k_3}\sum_{i=1}^n (\p_{x_i}P_{k_1}v
\p_{x_i}P_{k_2}w)]}_{L_{t}^2L_x^{\frac{2}{1+\e}}}\\
\les& \norm{u}_{Y^1} \sum_{k_1,k_2,k_3}2^{k_3\e}2^{k_1+k_2}\norm{P_{k_1}v
\p_{x_i}}_{L_{t}^2L_x^{\frac{4}{1+\e}}L_\theta^3}\norm{P_{k_2}w}_{L_{t}^2L_x^{\frac{4}{1+\e}}L_\theta^\infty}\\
\les& \norm{u}_{Y^1}\norm{v}_{Z^1}\norm{w}_{Z^1}.
\end{align*}
Therefore we complete the proof.
\end{proof}

\begin{lem}
We have
\begin{align}
\norm{u\sum_{i=1}^2(\p_{x_i}v \partial_{x_i}w)}_{W^{1}}\les&
\norm{u}_{Y^{1}}\norm{v}_{Z^{1}}\norm{w}_{Z^{1}}.
\end{align}
\end{lem}
\begin{proof}
By the definition, we have
\begin{align}\label{eq:tripf1}
\norm{u \sum_{i=1}^2\p_{x_i}v
\partial_{x_i}w}_{W^{1}}\leq&\sum_{k_i}2^{k_4}\sum_{j=0,1}\norm{\p_\theta^jP_{k_4}[P_{k_1}u\sum_{i=1}^2(P_{k_2}\p_{x_i}v
\partial_{x_i}P_{k_3}w)]}_{W_{k_4}}.
\end{align}
The $L_{t,x}^2$ component in $W_{k_4}$ is handled by the previous lemma.  So we only need  to handle the other component.  We assume $j=0$ in the above inequality since the other case $j=1$ is similar. 
By symmetry we may assume $k_2\leq k_3$ in the above summation. If
in the above summation we assume $k_4\leq k_1+40$, then
\begin{align*}
\eqref{eq:tripf1}\les&
\sum_{k_i}2^{k_4}\norm{P_{k_4}[P_{k_1}uP_{k_2}\p_{x_i}\bar v
\partial_{x_i}P_{k_3}w]}_{L_{t,x}^{4/3}}\\
\les& \sum_{k_i}2^{k_1}\norm{P_{k_1}u}_{L_{x,t}^{4}}2^{k_2}\norm{P_{k_2}v}_{L_{x,t}^{4}}2^{k_3}\norm{P_{k_3}w}_{L_{x,t}^{4}}\\
\les&\norm{u}_{Y^{1}}\norm{v}_{Z^{1}}\norm{w}_{Z^{1}}.
\end{align*}
Thus we assume $k_4\geq k_1+40$ in the summation of
\eqref{eq:tripf1}. We bound the summation case by case.

{\bf Case 1:} $k_2\leq k_1+20$

In this case we have $k_4\geq k_2+20$ and hence $|k_4-k_3|\leq 5$.
By Lemma \ref{lem:L2est} we get
\begin{align*}
\eqref{eq:tripf1}\les&
\sum_{k_i}2^{k_4}2^{-k_4/2}\norm{P_{k_4}[P_{k_1}uP_{k_2}\p_{x_i}\bar
v
\partial_{x_i}P_{k_3}w]}_{L_{\ve{e}}^{1,2}}\\
\les& \sum_{k_i}2^{k_4}2^{-k_4/2}\norm{P_{k_1}uP_{k_3}\partial_{x_i}w}_{L_{x,t}^{2}}\norm{P_{k_2}\p_{x_i}v}_{L_{\ve{e}}^{2,\infty}}\\
\les&\norm{u}_{Y^{1}}\norm{v}_{Z^{1}}\norm{w}_{Z^{1}}.
\end{align*}

{\bf Case 2:} $k_2\geq k_1+21$

In this case we have $k_4\leq k_3+40$. Let
$g=\sum_{i=1}^n(P_{k_2}\p_{x_i}v\cdot P_{k_3}\partial_{x_i}w)$. Then
we have
\begin{align*}
\eqref{eq:tripf1}\les&\sum_{k_i}2^{k_4}\norm{P_{k_4}[P_{k_1}uQ_{\leq k_2+k_3}g]}_{N_{k_4}}+\sum_{k_i}2^{k_4}\norm{P_{k_4}[P_{k_1}uQ_{\geq k_2+k_3}g]}_{N_{k_4}}\\
:=&I+II.
\end{align*}
First we estimate the term $II$.  We have
\begin{align*}
II\les &\sum_{k_i}2^{k_4}\norm{P_{k_4}[P_{k_1}Q_{\geq k_2+k_3-10}u\cdot Q_{\geq k_2+k_3}g]}_{N_{k_4}}\\
&+\sum_{k_i}2^{k_4}\norm{P_{k_4}[P_{k_1}Q_{\leq k_2+k_3-10}u\cdot Q_{\geq k_2+k_3}g]}_{N_{k_4}}\\
:=&II_1+II_2.
\end{align*}
For the term $II_1$ we have
\begin{align*}
II_1\les &\sum_{k_i}2^{k_4}\norm{P_{k_4}[P_{k_1}Q_{\geq k_2+k_3-10}u\cdot Q_{\geq k_2+k_3}g]}_{L_t^1L_x^2}\\
\les &\sum_{k_i}2^{k_4}\norm{P_{k_1}Q_{\geq k_2+k_3-10}u}_{L_t^2L_x^\infty}\norm{Q_{\geq k_2+k_3}g]}_{L_t^2L_x^2}\\
\les &\sum_{k_i}2^{k_4}2^{k_1}\norm{P_{k_1}Q_{\geq k_2+k_3-10}u}_{L_t^2L_x^2}\norm{Q_{\geq k_2+k_3}g]}_{L_t^2L_x^2}\\
\les &\sum_{k_i}2^{k_4}2^{k_1}2^{-(k_2+k_3)}\norm{P_{k_1}u}_{X^{0,1}+\bar X^{0,1}}2^{(k_2-k_3)/2}2^{k_2+k_3}\norm{P_{k_2}v}_{F_{k_2}}\norm{P_{k_3}w}_{F_{k_3}}\\
\les&\norm{u}_{Y^{1}}\norm{v}_{Z^{1}}\norm{w}_{Z^{1}}.
\end{align*}
For the term $II_2$, since $k_4\geq k_1+40$, then we may assume $g$
has frequency of size $2^{k_4}$.  The resonance function in the
product $P_{k_1}u\cdot p_{k_4}g$ is of size $\les 2^{k_1+k_4}$.
Thus the output modulation is of size $\ges 2^{k_2+k_3}$.  Then we
get
\begin{align*}
II_2\les &\sum_{k_i}2^{k_4}2^{-(k_2+k_3)/2}\norm{P_{k_4}[P_{k_1}Q_{\leq k_2+k_3-10}u\cdot Q_{\geq k_2+k_3}g]}_{L^2_{t,x}}\\
\les &\sum_{k_i}2^{k_4}2^{-(k_2+k_3)/2}2^{k_1}\norm{P_{k_1}u}_{L_t^\infty L_x^2}\cdot \norm{g}_{L^2_{t,x}}\\
\les&\sum_{k_i}2^{k_4}2^{-(k_2+k_3)/2}2^{k_1}2^{(k_2-k_3)/2}2^{k_2+k_3}\norm{P_{k_1}u}_{Y_{k_1}}\norm{P_{k_2}v}_{F_{k_2}}\norm{P_{k_3}w}_{F_{k_3}}\\
\les&\norm{u}_{Y^{1}}\norm{v}_{Z^{1}}\norm{w}_{Z^{1}}.
\end{align*}

Now we estimate the term $I$.  We have
\begin{align*}
I\les &\sum_{k_i}2^{k_4}\norm{P_{k_4}[P_{k_1}u\cdot Q_{\leq
k_2+k_3}\sum_{i=1}^2(P_{k_2}\p_{x_i}Q_{\geq k_2+k_3+40}v\cdot
P_{k_3}\partial_{x_i}w)]}_{N_{k_4}}\\
&+\sum_{k_i}2^{k_4}\norm{P_{k_4}[P_{k_1}u\cdot Q_{\leq
k_2+k_3}\sum_{i=1}^2(P_{k_2}\p_{x_i}Q_{\leq k_2+k_3+39}v\cdot
P_{k_3}\partial_{x_i}w)]}_{N_{k_4}}\\
:=&I_1+I_2.
\end{align*}
For the term $I_1$, since the resonance function in the product
$P_{k_2}v\cdot p_{k_3}w$ is of size $\les 2^{k_2+k_3}$, then we may
assume $P_{k_3}w$ has modulation of size $\ges 2^{k_2+k_3}$.  Then
we get
\begin{align*}
I_1\les& \sum_{k_i}2^{k_4}2^{-k_4/2}\norm{P_{k_4}[P_{k_1}u\cdot
Q_{\leq k_2+k_3}\sum_{i=1}^2(P_{k_2}\p_{x_i}Q_{\geq
k_2+k_3+40}v\cdot
P_{k_3}\partial_{x_i}Q_{\geq k_2+k_3-5}w)]}_{L^{1,2}_{\ve e}}\\
\les&\sum_{k_i}2^{k_4}2^{-k_4/2}\norm{P_{k_1}u}_{L^\infty_{t,x}}2^{k_2+k_3}\norm{P_{k_2}v}_{L^{2,\infty}_{\ve e}}\norm{P_{k_3}Q_{\geq k_2+k_3-5}w}_{L^2_{t,x}}\\
\les&\sum_{k_i}2^{k_4}2^{-k_4/2}2^{(k_2+k_3)/2}2^{k_2/2}2^{k_1}\norm{P_{k_1}u}_{Y_{k_1}}\norm{P_{k_2}v}_{F_{k_2}}\norm{P_{k_3}w}_{F_{k_3}}\\
\les&\norm{u}_{Y^{1}}\norm{v}_{Z^{1}}\norm{w}_{Z^{1}}.
\end{align*}

Finally, we estimate the term $I_2$.  For this term, we need to use
the null structure observed by Bejenaru \cite{Bej}.  We can rewrite
\begin{align}
2\nabla u\cdot \nabla v=(i\p_t-\Delta)u\cdot v+u\cdot
(i\p_t-\Delta)v-(i\p_t-\Delta)(u\cdot v).
\end{align}
Let $L=i\p_t-\Delta$. Then we have
\begin{align*}
I_2=&\sum_{k_i}2^{k_4}\norm{P_{k_4}[P_{k_1}u\cdot Q_{\leq
k_2+k_3}(P_{k_2}LQ_{\leq k_2+k_3+39}v\cdot
P_{k_3}w)]}_{N_{k_4}}\\
&+\sum_{k_i}2^{k_4}\norm{P_{k_4}[P_{k_1}u\cdot Q_{\leq
k_2+k_3}(P_{k_2}Q_{\leq k_2+k_3+39}v\cdot
P_{k_3}Lw)]}_{N_{k_4}}\\
&+\sum_{k_i}2^{k_4}\norm{P_{k_4}[P_{k_1}u\cdot Q_{\leq
k_2+k_3}L(P_{k_2}Q_{\leq k_2+k_3+39}v\cdot
P_{k_3}w)]}_{N_{k_4}}\\
:=&I_{21}+I_{22}+I_{23}.
\end{align*}
For the term $I_{21}$, if $k_4\geq k_2+10$, then $|k_4-k_3|\leq 5$ and hence
\begin{align*}
I_{21}\les& \sum_{k_i}2^{k_4}2^{-k_4/2}\norm{P_{k_4}[P_{k_1}u\cdot Q_{\leq
k_2+k_3}(P_{k_2}LQ_{\leq k_2+k_3+39}v\cdot P_{k_3}w)]}_{L^{1,2}_{\ve e}}\\
\les&\sum_{k_i}2^{k_4}2^{k_1}2^{-k_4/2}\norm{P_{k_1}u}_{L_t^\infty
L_x^2}\norm{P_{k_2}Lv}_{L^2_{t,x}}\norm{P_{k_3}w}_{L^{2,\infty}_{\ve e}}\\
\les&\norm{u}_{Y^{1}}\norm{v}_{Z^{1}}\norm{w}_{Z^{1}}.
\end{align*}
On the other hand, if $k_4\leq k_2+10$, we get
\begin{align*}
I_{21}\les& \sum_{k_i}2^{k_4}\norm{P_{k_4}[P_{k_1}u\cdot Q_{\leq
k_2+k_3}(P_{k_2}LQ_{\leq k_2+k_3+39}v\cdot P_{k_3}w)]}_{L^{4/3}_{t,x}}\\
\les&\sum_{k_i}2^{k_4}2^{k_1}\norm{P_{k_1}u}_{L_t^\infty
L_x^2}\norm{P_{k_2}Lv}_{L^2_{t,x}}\norm{P_{k_3}w}_{L^{4}_{t,x}}\\
\les&\norm{u}_{Y^{1}}\norm{v}_{Z^{1}}\norm{w}_{Z^{1}}.
\end{align*}
For the term $I_{22}$, we may assume  $w$ has modulation $\les 2^{k_2+k_3}$. Then we get
\begin{align*}
I_{22}\les&\sum_{k_i}2^{k_4}2^{-k_4/2}\norm{P_{k_4}[P_{k_1}u\cdot
Q_{\leq k_2+k_3}(P_{k_2}Q_{\leq k_2+k_3+39}v\cdot
P_{k_3}Q_{\leq k_2+k_3+100}Lw)]}_{L^{1,2}_{\ve e}}\\
\les&\sum_{k_i}2^{k_4}2^{-k_4/2}2^{k_1}\norm{P_{k_1}u}_{L_t^\infty
L_x^2}\norm{P_{k_2}v}_{L^{2,\infty}_{\ve e}}\norm{
P_{k_3}Q_{\leq k_2+k_3+100}Lw)]}_{L^{2}_{t,x}}\\
\les&\sum_{k_i}2^{k_4}2^{-k_4/2}2^{k_1}2^{k_2/2}2^{(k_2+k_3)/2}\norm{P_{k_1}u}_{Y_{k_1}}\norm{P_{k_2}v}_{F_{k_2}}\norm{P_{k_3}w}_{X^{0,1/2,\infty}}\\
\les&\norm{u}_{Y^{1}}\norm{v}_{Z^{1}}\norm{w}_{Z^{1}}.
\end{align*}
Next we estimate the term $I_{23}$. We have
\begin{align*}
I_{23}\les& \sum_{k_i}2^{k_4}\norm{P_{k_4}[P_{k_1}u\cdot
Q_{[k_1+k_4+100, k_2+k_3]}L(P_{k_2}Q_{\leq k_2+k_3+39}v\cdot
P_{k_3}w)]}_{N_{k_4}} \\
&+ \sum_{k_i}2^{k_4}\norm{P_{k_4}[P_{k_1}u\cdot Q_{\leq
k_1+k_4+99}L(P_{k_2}Q_{\leq k_2+k_3+39}v\cdot
P_{k_3}w)]}_{N_{k_4}} \\
:=&I_{231}+I_{232}.
\end{align*}
For the term $I_{232}$ we have
\begin{align*}
I_{232}\les& \sum_{k_i}2^{k_4}2^{-k_4/2}\norm{P_{k_4}[P_{k_1}u\cdot
Q_{\leq k_1+k_4+99}L(P_{k_2}Q_{\leq k_2+k_3+39}v\cdot
P_{k_3}w)]}_{L^{1,2}_{\ve e}} \\
\les & \sum_{k_i}2^{k_4}2^{-k_4/2}\norm{P_{k_1}u}_{L^{2,\infty}_{\ve e}}2^{k_1+k_4}2^{(k_2-k_3)/2}\norm{P_{k_2}v}_{F_{k_2}}\norm{P_{k_3}w}_{F_{k_3}}\\
\les&\norm{u}_{Y^{1}}\norm{v}_{Z^{1}}\norm{w}_{Z^{1}}.
\end{align*}
For the term $I_{231}$ we have
\begin{align*}
I_{231}\les&
\sum_{k_i}\sum_{j_2=k_1+k_4+100}^{k_2+k_3}2^{k_4}\norm{P_{k_4}Q_{\leq
j_2-10}[P_{k_1}u\cdot Q_{j_2}L(P_{k_2}Q_{\leq k_2+k_3+39}v\cdot
P_{k_3}w)]}_{N_{k_4}}\\
&+\sum_{k_i}\sum_{j_2=k_1+k_4+100}^{k_2+k_3}2^{k_4}\norm{P_{k_4}Q_{\geq
j_2-9}[P_{k_1}u\cdot Q_{j_2}L(P_{k_2}Q_{\leq k_2+k_3+39}v\cdot
P_{k_3}w)]}_{N_{k_4}}\\
:=&I_{2311}+I_{2312}.
\end{align*}
For the term $I_{2312}$ we have
\begin{align*}
I_{2312}\les&
\sum_{k_i}\sum_{j_2=k_1+k_4+100}^{k_2+k_3}\sum_{j_3\geq
k_2-9}2^{k_4}2^{-j_3/2}\norm{P_{k_4}Q_{j_3}[P_{k_1}u\cdot
Q_{j_2}L(P_{k_2}Q_{\leq k_2+k_3+39}v\cdot
P_{k_3}w)]}_{L^2_{t,x}}\\
\les& \sum_{k_i}2^{k_4}2^{k_1}2^{(k_2+k_3)/2}2^{(k_2-k_3)/2}\norm{P_{k_1}u}_{Y_{k_1}}\norm{P_{k_2}v}_{F_{k_2}}\norm{P_{k_3}w}_{F_{k_3}}\\
\les&\norm{u}_{Y^{1}}\norm{v}_{Z^{1}}\norm{w}_{Z^{1}}.
\end{align*}
For the term $I_{2311}$ we have
\begin{align*}
I_{2311}\les&
\sum_{k_i}\sum_{j_2=k_1+k_4+100}^{k_2+k_3}2^{k_4}\norm{P_{k_4}Q_{\leq
j_2-10}[P_{k_1}\tilde Q_{j_2}u\cdot Q_{j_2}L(P_{k_2}Q_{\leq
k_2+k_3+39}v\cdot
P_{k_3}w)]}_{L_t^1L_x^2}\\
\les&\sum_{k_i}\sum_{j_2=k_1+k_4+100}^{k_2+k_3}2^{k_4}2^{k_1}\norm{P_{k_1}\tilde Q_{j_2}u}_{L^2_{t,x}}2^{j_2}2^{(k_2-k_3)/2}\norm{P_{k_2}v}_{F_{k_2}}\norm{P_{k_3}w}_{F_{k_3}}\\
\les&\norm{u}_{Y^{1}}\norm{v}_{Z^{1}}\norm{w}_{Z^{1}}.
\end{align*}
Therefore, we complete the proof.
\end{proof}

\subsection*{Acknowledgment}
The author would like to thank Victor Lie for a helpful discussion
on the maximal function estimate, and Ioan Bejenaru on the null form.
This work is supported in part by
NNSF of China (No.11371037), Beijing
Higher Education Young Elite Teacher Project (No.\ YETP0002), Fok
Ying Tong education foundation (No.\ 141003), and ARC Discovery grants number DP130101302.

\end{document}